\documentclass[11pt]{article}

\usepackage[margin=1in]{geometry}
\usepackage{authblk}

\usepackage[colorlinks,citecolor=blue,urlcolor=blue]{hyperref}
\usepackage{amsthm,amsmath,amsfonts,amssymb,graphicx,subcaption,soul,multicol,url,comment,arydshln,color,cleveref}
\usepackage[bb=boondox]{mathalfa}
\usepackage{enumerate}
\usepackage[shortlabels]{enumitem}

\newtheorem{thm}{Theorem}
\newtheorem{cor}{Corollary}
\newtheorem{lem}{Lemma}
\newtheorem{prop}{Proposition}
\newtheorem{defn}{Definition}
\newtheorem{assum}{Assumption}

\newtheorem{rem}{Remark}

\long\def\symbolfootnote[#1]#2{\begingroup%
\def\thefootnote{\fnsymbol{footnote}}\footnote[#1]{#2}\endgroup}

\newcommand{\defeq}{\stackrel{\text{def}}{=}}
\newcommand{\norm}[1]{\left\lVert#1\right\rVert}

\newcommand{\abs}[1]{\left\lvert#1\right\rvert}

\title{Data-driven robust Markov decision processes on Borel spaces: performance guarantees via an axiomatic approach}

\author{Sivaramakrishnan Ramani}
\affil{\small Department of Computational Applied Mathematics and Operations Research, Rice University, Houston, USA. sivar@rice.edu}

\begin{document}

\date{}
\maketitle

\begin{abstract}
We consider Markov decision processes (MDPs) with unknown disturbance distribution and address this problem using the robust Markov decision process (RMDP) approach. We construct the empirical distribution of the unknown disturbance distribution and characterize our ambiguity set of distributions as the sublevel set of a nonnegative distance function from the empirical distribution. By connecting the weak convergence of distributions to convergence with respect to the distance function, we prove that the robust optimal value function and the out-of-sample value function converge to the true optimal value function with increasing sample-sizes. We establish that, for finite sample-sizes, the robust optimal value function serves as a high probability upper bound on the out-of-sample value function. We also obtain probabilistic convergence rates, sample complexity bounds, and out-of-distribution performance bounds. The finite sample performance guarantees rely on the distance function satisfying a certain concentration type inequality. Several well-studied distances in the literature meet the requirements imposed on the distance function. We also analyze the data-driven properties of empirical MDPs and demonstrate that, unlike our data-driven RMDPs, empirical MDPs fail to satisfy some of the finite sample performance guarantees.
\end{abstract}

\section{Introduction}\label{sec:introduction}
A discrete-time infinite-horizon Markov decision process (MDP) can be described as follows \cite{bertsekas1978stochastic,hernandez2012discrete}. A system is in state $x_t \in \mathbb{X}$ at the beginning of stage $t \in \{1,2,\ldots,\}$. Upon observing this state of the system, the decision-maker chooses an action $a_t \in A(x_t) \subseteq \mathbb{A}$. The system then makes a stochastic transition to the state $x_{t+1} = F(x_t,a_t,w_t)$, and the decision-maker incurs a cost $c(x_t,a_t,w_t) \in \mathbb{R}$. Here, $\{w_t\}$ is a sequence of independent and identically distributed (iid) random variables with values in the disturbance space $\mathbb{W}$ and distribution $\mu$. Denote the set of history-dependent randomized policies as $\Pi_{\text{HR}}$. The value of the policy $\pi \in \Pi_{\text{HR}}$ when the initial state is $x$ equals
\begin{equation}\label{eqn:mdpvalue}
 J(\pi,x) \defeq \mathbb{E}^\pi\left[\sum_{t=1}^\infty \alpha^{t-1}c(x_t,a_t,w_t)\Big|x_1=x\right],
\end{equation}
where $\alpha \in (0,1)$ is the discount factor. Here, $\mathbb{E}^\pi$ is the expectation taken with respect to the distribution induced by the policy $\pi$. The value of the MDP is then given as
\begin{equation}\label{eqn:mdpoptimality}
J^*(x) \defeq \inf_{\pi \in \Pi_{\text{HR}}}J(\pi,x) \ \forall x \in \mathbb{X}. 
\end{equation}
The function $J^*$ is called the optimal value function and a policy $\pi \in \Pi_{\text{HR}}$ for which $J(\pi,x) = J^*(x) \ \forall x \in \mathbb{X}$ is called the optimal policy.  A deterministic stationary policy $\pi$ is a measurable function from $\mathbb{X}$ to $\mathbb{A}$ such that $\pi(x) \in A(x)$ is the action prescribed in state $x \in \mathbb{X}$, regardless of stage $t \in \{1,2,\ldots,\}$. The set of deterministic stationary policies is denoted by $\Pi$. Sufficient conditions that guarantee the existence of an optimal policy in $\Pi$ will be assumed in \Cref{sec:rmdp_model} and further sections.

The distribution $\mu$ of the iid disturbance sequence $\{w_t\}$ is unknown most of the time.  Robust Markov decision processes (RMDPs) tackle this issue using a worst-case approach. In RMDPs, the decision-maker models the unknown distribution of the disturbance variable to belong to a set of distributions, referred the  ambiguity set, and the aim is to minimize the largest expected total discounted cost over all distributions from this ambiguity set.  This can be modeled as a two-player Markov game (or stochastic game) with complete information between the decision-maker and a fictitious adversary, where the decision-maker picks a sequence of actions to minimize the expected total discounted cost, while the adversary selects a sequence of distributions from the ambiguity set that maximizes the expected total discounted cost incurred by the decision-maker. For RMDPs on general spaces, the adversary is allowed to employ history-dependent randomized policies, and various properties of the RMDP under consideration are then established \cite{bauerle2022distributionally,glauner2020robust,gonzalez2003minimax,jaskiewicz2011stochastic,jaskiewicz2014robust}. In contrast, the adversary's policy class for RMDPs over finite spaces is typically restricted to deterministic Markovian policies (or a subclass of it), and the corresponding RMDP problem is then analyzed \cite{goyal2023robust,iyengar2005robust,mannor2016robust,nilim2005robust,wiesemann2013robust}. An exception in the case of finite space RMDPs is  \cite{wang2023foundation}, which formally analyzes RMDPs where the adversary can choose a policy from the set of history-dependent randomized policies.

\subsection{Contributions of this paper}\label{sec:contributions}
In this paper, we consider RMDPs with ambiguity sets of the form
\begin{equation}\label{eqn:ambiguity_set}
 {\cal P}^N(\epsilon) \defeq \left\{\nu \in {\cal M}(\mathbb{W}) \vert d(\nu, \hat\mu^N) \leq \epsilon\right\}.
\end{equation}
Here, $\hat\mu^N \defeq (1/N)\sum_{i=1}^N \delta_{w_i}$  is the empirical distribution of $\mu$   constructed based on iid samples $w_1,\ldots,w_N$ of the disturbance variable $w$. The notation $\delta_{w_i}$ denotes the Dirac measure concentrated at $w_i$ and ${\cal M}(\mathbb{W})$ is the set of probability distributions on $\mathbb{W}$.  We refer to the parameter $\epsilon \in [0,\infty)$ as the ambiguity set radius. The function $d: {\cal M}(\mathbb{W}) \times {\cal M}(\mathbb{W}) \to [0,\infty]$ is a nonnegative extended real-valued distance function (which may not be a metric) over the space of probability distributions. Thus, ambiguity sets ${\cal P}^N(\epsilon)$ equal the sublevel sets of the distance function from the empirical distribution. Throughout this work, we assume that ${\cal P}^N(\epsilon)$ is  nonempty with probability $1$ for all $N \geq 1$ and $\epsilon \in [0,\infty)$. The condition $d(\nu,\nu) = 0$ for all $\nu \in {\cal M}(\mathbb{W})$ ensures the nonemptiness property. This condition on the distance function is satisfied by all the  distances listed in  \Cref{sec:controldist}. These include total variation (TV) distance, Hellinger distance, Kullback-Leibler (KL) divergence, $\chi^2$ distance, Wasserstein distance, bounded Lipschitz metric, and Prokhorov metric.

We formulate our RMDP problem using the framework developed in \cite{gonzalez2003minimax}.  For $t \in \mathbb{N}$, let $h_t = (x_1,a_1,\nu_1,\ldots, x_{t-1},a_{t-1},\nu_{t-1}, x_t)$ denote a history sequence up to stage $t$, where $x_k \in \mathbb{X}, a_k \in A(x_k)$, and $\nu_k \in {\cal P}^{N}(\epsilon)$ for all $k \in \mathbb{N}$. Denote the set of all history sequences up to stage $t$ by $H_t$. A history-dependent randomized policy $\pi=(\pi_1,\pi_2,\ldots)$ for the decision-maker is a sequence of stochastic kernels from $H_t$ to $\mathbb{A}$ satisfying $\pi_t(A(x_t)|h_t) = 1$ for all $h_t \in H_t$. Let $h_t^{\#} = (h_t,a_t)$ be an extended history sequence where $a_t \in A(x_t)$. Denote by $H_t^{\#}$ the set of extended history sequences up to stage $t$. A history-dependent randomized policy $\gamma = (\gamma_1,\gamma_2,\ldots)$ for the adversary is a sequence of stochastic kernels from $H_t^{\#}$ to ${\cal P}^{N}(\epsilon)$. Denote the set of history-dependent randomized policies for the decision maker and the adversary by $\Pi_{\text{HR}}^{N,\epsilon}$ and $\Gamma_{\text{HR}}^{N,\epsilon}$, respectively.  Supposing the initial state of the system is $x$, the expected total discounted cost incurred by a pair $\pi \in \Pi_{\text{HR}}^{N,\epsilon}$ and $\gamma \in \Gamma_{\text{HR}}^{N,\epsilon}$ is given by $\tilde J_{\#}^{N,\epsilon}(\pi,\gamma,x) \defeq \mathbb{E}^{\pi,\gamma}\left[\sum_{t=1}^{\infty}\alpha^{t-1}c(x_t,a_t,w_t)\vert x_1 = x \right]$,
where $\mathbb{E}^{\pi,\gamma}$ denotes the expectation  taken with respect to the distribution induced by $\pi$ and $\gamma$. The value of the corresponding RMDP is then defined as 
\begin{equation}\label{eqn:robustmdp}
\tilde J^{N,\epsilon}(x) \defeq \inf_{\pi \in \Pi^{N,\epsilon}_{\text{HR}}}\sup_{\gamma \in \Gamma^{N,\epsilon}_{\text{HR}}}\tilde J_{\#}^{N,\epsilon}(\pi,\gamma,x) \ \forall x \in \mathbb{X}.
\end{equation}
We refer to \eqref{eqn:robustmdp} as the data-driven distance-based RMDP problem (henceforth, data-driven RMDP problem). We call $\tilde J^{N,\epsilon}$ the data-driven robust optimal value function (henceforth, robust optimal value function), and a policy $\hat\pi^N \in \Pi^{N,\epsilon}_{\text{HR}}$  such that $\sup_{\gamma \in \Gamma^{N,\epsilon}_{\text{HR}}}\tilde J_{\#}^{N,\epsilon}(\hat\pi^N,\gamma,x) = \tilde J^{N,\epsilon}(x) \ \forall x \in \mathbb{X}$ as the data-driven robust optimal policy (henceforth, robust optimal policy). A deterministic stationary policy for the decision-maker is a measurable function $\pi: \mathbb{X} \to \mathbb{A}$ such that $\pi(x) \in A(x)$ is the action played in state $x \in \mathbb{X}$ for all time-stages $t \in \left\{1,2,\ldots,\right\}$. The model components of our data-driven RMDP will be the same as the underlying  MDP instance, except that the disturbance distribution will be replaced with an ambiguity set. Hence, the set of deterministic stationary policies for the decision-maker in our RMDP coincides with the set of deterministic stationary policies, $\Pi$, of the considered MDP instance. Sufficient conditions for the existence of a robust optimal policy in $\Pi$ will be assumed in \Cref{sec:rmdp_model} and beyond. 

\Cref{sec:data-driven_rmdp} analyzes the asymptotic (as sample-size $N \to \infty$) and finite-sample properties of our data-driven RMDPs under suitable conditions on the distance function. More precisely, we prove that our data-driven RMDPs satisfy the following data-driven performance guarantees.

{\bf 1. Asymptotic convergence to the true optimal value.} We prove that the robust optimal value function converges almost surely to the optimal value function of the true MDP, as $N \to \infty$. Next, we prove the almost sure convergence of the out-of-sample value function to the true optimal value function, as $N \to \infty$. The out-of-sample value function is the value of the robust optimal policy evaluated under the true distribution of the random disturbance. Consistent with the notation in \eqref{eqn:mdpvalue}, we denote the out-of-sample value function by $J(\hat\pi^N, x)$ for all $x \in \mathbb{X}$. Both these convergence results are proved in  \Cref{sec:valueconv} under \Cref{assum:distancemetrize}. That assumption states that if the distributions converge with respect to the distance function, then they weakly converge. An easily verifiable sufficient condition for \Cref{assum:distancemetrize} is presented in \Cref{assum:betametrize}.

{\bf 2. Probabilistic performance guarantee on the out-of-sample value function.} We prove in \Cref{sec:probguarantee} that for finite sample-sizes, the out-of-sample value function is bounded above by the robust optimal value function, with a high probability. This result can be viewed as construction of a confidence interval for the out-of-sample value function. This claim is established under \Cref{assum:controlconc} which is a certain concentration inequality for the distance function.  

{\bf 3. Probabilistic convergence rate and sample complexity.} Under a minor strengthening of \Cref{assum:betametrize}, we establish in \Cref{sec:conv_rate} the rate of convergence of the robust and out-of-sample value functions to the true optimal value function. The sample complexity is defined as the minimum number of iid samples of the disturbance variable needed to bound the suboptimality gap of the robust optimal policy (on the true MDP) to a desired accuracy with a specified level of confidence. In \Cref{sec:sample_complexity}, we derive a lower bound on the sample complexity  using the probabilistic convergence rate result. 

Apart from the assumptions on the distance function, another crucial requirement to establish our data-driven performance guarantees is the suitable selection of the ambiguity set radius. We outline conditions on the ambiguity set radius and, whenever feasible, present closed-form expressions for the ambiguity set radius that ensure the attainment of the data-driven performance guarantees. 

We note that the probabilistic performance guarantee on the out-of-sample value function and the probabilistic convergence rate provide different utilities for a decision-maker. The probabilistic convergence rate provides a probabilistic bound on the closeness of the out-of-sample value function  to the true optimal value function. Since the distribution of the  disturbance variable is unknown, we cannot compute the out-of-sample value function as well as the true optimal value function. Hence, the probabilistic convergence rate result does not provide insight regarding the values of the out-of-sample value function. The probabilistic performance guarantee on the out-of-sample value function fills this gap by providing a high probability computable upper bound on the out-of-sample value function, and  this guarantee can aid the decision-maker in other planning and downstream tasks.  

\Cref{sec:out_of_dist_bounds} analyzes the performance loss
due to deploying a robust optimal policy in a MDP whose disturbance distribution is different from the distribution used to obtain iid samples. Motivating scenarios for studying such performance losses arise across domains such as robotics \cite{wagenmaker2024overcoming} and engineering system design \cite{degrave2022magnetic}, among others. \Cref{thm:ood_rate} provides probabilistic bounds on the out-of-distribution performance loss of the robust optimal policy. The bound reveals that the performance loss decomposes into a statistical loss that typically vanishes with increasing sample-sizes,  and a nonstatistical loss that captures the mismatch (in a certain sense) between the  distribution from which samples are drawn and the new disturbance distribution. 

Empirical MDPs provide an alternative approach to handle the unknown disturbance distribution, in which the MDP under consideration is solved by replacing the unknown disturbance distribution  with its empirical counterpart. \Cref{sec:empirical_mdps} analyzes the quality of the empirical optimal value function (optimal value function of the empirical MDP) and the empirical optimal policy (optimal policy of the empirical MDP) using two metrics --- sample complexity and probabilistic performance guarantee on the out-of-sample value of the empirical optimal policy. \Cref{thm:empirical_mdps_counterexample} in that section establishes via a counterexample that for all confidence levels beyond a threshold and for every finite sample-size, the empirical optimal value function fails to serve as an upper bound on the value of the empirical optimal policy on the true MDP. Additionaly,  \Cref{sec:empirical_mdps}  discusses in a precise sense the ramification of the claims established in   \Cref{thm:empirical_mdps_counterexample}. Roughly speaking, it highlights the inability to bound the suboptimality gap of the empirical optimal policy to any desired level of accuracy and confidence while simultaneously attempting to maximize the probability that the out-of-sample value function of the empirical optimal policy is upper bounded by the empirical optimal value function. This reveals a  divergence between empirical MDPs and our data-driven RMDPs, where all finite sample performance guarantees simultaneously hold for our data-driven RMDPs (see the last paragraph of \Cref{sec:sample_complexity}).

\Cref{sec:controldist} lists well-known distance functions for which  \Cref{assum:distancemetrize} and \Cref{assum:controlconc} hold. These distances include TV distance, Hellinger distance, KL divergence, $\chi^2$ distance, Wasserstein distance, bounded Lipschitz metric, and Prokhorov metric. As a result, the data-driven RMDPs characterized by  ambiguity sets constructed using these distances satisfy the data-driven performance guarantees outlined in this paper.

\subsection{Literature review}\label{sec:lit_review}
A detailed literature review of RMDPs on finite spaces is available in \cite{ramani2022robust,ramani2024family}. Formal analyses of RMDPs on Borel spaces, including a characterization of various optimality conditions, appear in  \cite{bauerle2022distributionally,glauner2020robust,gonzalez2003minimax,jaskiewicz2011stochastic,jaskiewicz2014robust}. RMDPs with distance-based ambiguity sets, especially the Wasserstein ambiguity set, are well-studied in the literature; see   \cite{hakobyan2021wasserstein,hakobyan2023distributionally,kim2023distributional,neufeld2023bounding,tzortis2015tv,tzortis2019tv,tzortzis2016robust,yang2017distributionally,yang2020wasserstein} and references therein. Among them,  \cite{kim2023distributional,yang2020wasserstein} study data-driven RMDPs with Wasserstein ambiguity sets and establish probabilistic performance guarantee result on the out-of-sample value function. However, no asymptotic convergences or sample complexity results were established. Robust approaches similar to the one considered in this paper have also been studied in other related contexts such as partially observable systems, for example,  \cite{boskos2020data,hakobyan2024wasserstein,petersen2000minimax,ugrinovskii1999finite,ugrinovskii2002minimax,van2015distributionally}, and analyzing the connections between RMDPs and distributionally robust multistage stochastic programs \cite{shapiro2021distributionally,shapiro2022distributionally}, among others.

The problem of analyzing the performance loss incurred due to approximating the true MDP with a non-data-driven approximate model is well-studied in the literature. In this setup, the quality of approximation is typically analyzed via the absolute or relative difference between the optimal quantities of the approximate and the true model. See   \cite{bozkurt2025model,gordienko2000estimates,gordienko1998robustness,langen1981convergence,martinez2020stability,montes2003estimates,muller1997does,neufeld2023bounding,saldi2018finite} and references therein for work on this topic. The paper \cite{cooper2012performance} considers MDPs with unknown disturbance distribution and state-, action-, and disturbance-spaces that are subsets of the real line. They derive a sample complexity result by approximating their unknown MDP using empirical MDP. The empirical MDP approach was extended to arbitrary Borel spaces in \cite{gordienko2022stability,gordienko2008discounted,kara2020robustness,zhou2024robustness}, where  \cite{gordienko2022stability,gordienko2008discounted,kara2020robustness} establish the asymptotic convergence of the empirical optimal value function and the corresponding out-of-sample value function to the true optimal value function, and  \cite{zhou2024robustness} presents a sample complexity result.   

Recently, there has been an increased interest in the robust reinforcement learning problem \cite{clavier2024towards,panaganti2022sample,ramesh2024distributionally,shi2023curious,wang2025statistical,xu2023improved,yang2022toward}. In this problem, the goal is to solve an RMDP with a distance-based ambiguity set. The ambiguity set radius is fixed and known, while the nominal distribution (which serves as the center of the ambiguity set) is unknown. The nominal distribution is replaced with its empirical estimate, and the resulting empirical RMDP is solved. The statistical properties of the empirical RMDP in connection with the RMDP under consideration are then investigated. Although the robust reinforcement learning problem seems  similar to the data-driven RMDPs studied in this paper, we note that their setups are entirely different. This is because, for the data-driven RMDPs considered in this paper, the nominal problem is the minimization problem \eqref{eqn:mdpoptimality} rather than a minimax problem (RMDP), and the data-driven performance guarantees in relation to the minimization problem \eqref{eqn:mdpoptimality} are established. Hence, the data-driven performance guarantees established in this paper are not comparable with the results obtained in the literature for the robust reinforcement learning problem. Also, the papers on robust reinforcement learning focus only on ambiguity sets defined using TV, KL, $\chi^2$, Wasserstein, and $f_\alpha$-divergence with $\alpha > 1$. Those papers first reformulate the adversary's decision problem using convex programming duality,  and then establish the statistical properties by analyzing the reformulated problem. This contrasts the axiomatic approach pursued in this work, in which we rely on a topological-like characterization of the distance function (\Cref{assum:distancemetrize}) and a concentration type inequality (\Cref{assum:controlconc}) to establish the data-driven performance guarantees for a family of ambiguity sets.   

Our previous papers on data-driven RMDPs over finite state- and action-spaces \cite{ramani2022robust,ramani2024family} are the closest work to this paper. For the data-driven RMDPs considered in \cite{ramani2022robust,ramani2024family}, asymptotic and finite sample properties were established under conditions similar to \Cref{assum:distancemetrize} and \Cref{assum:controlconc} of this paper. The data-driven performance guarantees established in  \Cref{sec:data-driven_rmdp} of 
this paper can be viewed as an extension of those results to arbitrary Borel spaces. However, the analyses in this paper are distinct and quite sophisticated. Unlike \cite{ramani2022robust,ramani2024family}, the proofs in this paper rely on measure-theoretic concepts and results to establish the stated claims. In addition, the results in this paper rely on a suitable set of assumptions imposed on the components of the MDP model. This was not the case in \cite{ramani2022robust,ramani2024family}, where all data-driven performance guarantees were established without any restriction on the components of the MDP model.

\subsection{Notation and Preliminaries}\label{sec:notations_preliminaries}
Given a metric space $(X,\varrho)$, let $\Sigma_b(X)$ and $C_b(X)$ denote the set of real-valued bounded measurable and continuous functions over $X$, respectively.  For any $f \in \Sigma_b(X)$, the supremum norm $\norm{f}_\infty$ is defined as $\norm{f}_\infty \defeq \sup_{x \in X}\abs{f(x)}$. The Lipschitz seminorm is defined as $\norm{f}_L \defeq \sup_{x \neq y}\abs{f(x) - f(y)}/\varrho(x,y)$. Define $\norm{f}_{BL}$ as $\norm{f}_{BL} \defeq \norm{f}_L + \norm{f}_\infty$. A sequence of probability distributions $\mu_k$ in ${\cal M}(X)$ is  weakly convergent to $\mu \in {\cal M}(X)$ if $\int_X f(x)\mu_k(dx) \to \int_X f(x)\mu(dx)$ for all $f \in C_b(X)$. Define the bounded Lipschitz metric $\beta: {\cal M}(X) \times {\cal M}(X) \to \mathbb{R}_+$ as 
\begingroup
\allowdisplaybreaks
\begin{equation}\label{eqn:betametric}
\beta(\mu,\nu) \defeq \sup_{f:X \to \mathbb{R},\ \norm{f}_{BL} \leq 1}\abs{\int_X f(x)\mu(dx) - \int_X f(x) \nu(dx)}.
\end{equation}
\endgroup
If $X$ is a separable metric space, then $\beta$ defined in \eqref{eqn:betametric} is a metric on ${\cal M}(X)$ and metrizes the topology of weak convergence \cite[Proposition 11.3.2, Theorem 11.3.3]{dudley2018real}. For $\nu \in {\cal M}(X)$ and $N \in \mathbb{N}$, we denote by $\nu_\otimes^N$ the product measure $\otimes_{i=1}^N \nu$ on $X^N$. The notation $\nu_\otimes^\infty$ denotes the product measure $\otimes_{i=1}^\infty \nu$ on $\times_{i=1}^\infty X$. A set-valued mapping (also known as correspondence or multifunction) $\varphi$ from a metric space $X$ to another metric space $Y$ is such that $\varphi(x) \subseteq Y$ for all $x \in X$. The set-valued mapping $\varphi$ is compact-valued if $\varphi(x)$ is compact for all $x \in X$. A compact-valued mapping $\varphi$ is upper hemicontinuous at $x \in X$ if for any sequence $\{(x_k,y_k)\}\ \in \{(x,y)|x \in X, y \in \varphi(x)\}$ with $x_k \to x$, then the sequence $y_k$ has a limit point in $\varphi(x)$ \cite[Theorem 17.20]{aliprantis2006infinite}. It is lower hemicontinuous at $x \in X$ if for all $x_k \to x$ and $y \in \varphi(x)$, there exists a subsequence $x_{k_m}$ of $x_k$ and elements $y_m \in \varphi(x_{k_m})$ for each $m$ such that $y_m \to y$ \cite[Theorem 17.21]{aliprantis2006infinite}. The set-valued map $\varphi$ is continuous (also known as hemicontinuous) at $x \in X$ if it is upper hemicontinuous and lower hemicontinuous at $x$. It is continuous if it is continuous for all $x \in X$. The notation ${\bf 0}$ denotes the real-valued function over $X$ that is identically zero, i.e., ${\bf 0}(x) = 0$ for all $x \in X$.

\section{The robust Markov decision model}\label{sec:rmdp_model}
This section presents the optimality criterion for MDPs and RMDPs, and the associated sufficient conditions for the characterization of those optimality conditions. The results are standard in the literature and presented here in a form suitable for our data-driven setup. We take $\mathbb{X},\mathbb{A},\mathbb{W}$ to be Borel spaces (Borel subsets of complete and separable metric spaces) and endow ${\cal M}(\mathbb{W})$ with the topology of weak convergence. We assume that $\mathbb{K} \defeq \{(x,a)|x \in \mathbb{X}, a \in A(x)\}$ is a measurable subset of $\mathbb{X} \times \mathbb{A}$ and contains the graph of a measurable function from $\mathbb{X}$ to $A$, i.e., there exists a measurable function $f:\mathbb{X} \to \mathbb{A}$ such that $f(x) \in A(x)$ for all $x \in \mathbb{X}$.  
\begin{assum}\label{assum:mdpassum}
\normalfont
The following assumptions hold for the MDP model.
\begin{enumerate}[(a)]
\item The set-valued mapping $\mathbb{X} \ni x \mapsto A(x) \subseteq \mathbb{A}$ is compact-valued and continuous.
\item The system evolution function $F: \mathbb{K} \times  \mathbb{W} \to \mathbb{X}$ is continuous over $\mathbb{K} \times \mathbb{W}$.
\item The single-stage cost function $c: \mathbb{K} \times \mathbb{W} \to \mathbb{R}$ is continuous, bounded, and nonnegative over $\mathbb{K} \times \mathbb{W}$.
\end{enumerate}
\end{assum}

\begin{assum}\label{assum:ambiguitycompact}
\normalfont
The set $\left\{\nu \in {\cal M}(\mathbb{W}):d(\nu,\hat \nu) \leq \epsilon\right\}$ is $\sigma$-weakly compact for every $\hat \nu \in {\cal M}(\mathbb{W})$ and $\epsilon \in [0,\infty)$.
\end{assum}

The above two assumptions are adapted from  \cite[Assumption 3.1]{gonzalez2003minimax} and serve as a sufficient condition for optimality characterization of MDPs and RMDPs.  \Cref{assum:mdpassum} also plays a crucial role in establishing data-driven performance guarantees. The requirements on the components of the MDP model laid out in \Cref{assum:mdpassum} are satisfied by applications in a wide range of areas such as optimal asset selling \cite[Section 3.4]{bertsekas2012dynamic}, wind energy management \cite[Example 7.2]{bauerle2022distributionally}, mold level control \cite[Example 7.1]{gonzalez2003minimax}, and stochastic inventory control \cite[Section 3.2]{bertsekas2012dynamic}.  \Cref{assum:ambiguitycompact} is needed only by RMDPs to derive the robust Bellman equation and other related optimality conditions. \Cref{lem:ballcompactness} in  \Cref{sec:controldist} presents a sufficient condition on the distance function under which \Cref{assum:ambiguitycompact} holds. Several well-known distance functions studied in the literature  satisfy the condition stated in \Cref{lem:ballcompactness}. These include $f$-divergences, Wasserstein distance, bounded Lipschitz metric, and Prokhorov metric. 

The following theorem presents optimality  characterization for RMDPs. This theorem is a recasting of the optimality results derived in Theorems 4.1 and 4.2 in \cite{gonzalez2003minimax} for a more general minimax stochastic control problem. The claims of those two theorems in \cite{gonzalez2003minimax}  depend on the sufficient conditions stated in Assumption 3.1 in that paper. For our data-driven RMDPs, the sufficient conditions outlined in \cite{gonzalez2003minimax} are captured in  \Cref{assum:mdpassum}, \Cref{assum:ambiguitycompact}, and \Cref{lem:transitioncont} (established using   \Cref{assum:mdpassum}).  

\begin{thm}
\normalfont
\cite[Theorem 4.1, 4.2]{gonzalez2003minimax}
\label{thm:robustbellman}
Fix a sample-size $N \geq 1, \epsilon \in [0,\infty)$, and consider the data-driven RMDP problem in \eqref{eqn:robustmdp} (for a fixed realization of $\hat\mu^N$). For any $v \in \Sigma_b(\mathbb{X})$, the robust Bellman operator $\tilde \Phi^{N,\epsilon}$ on $\Sigma_b(\mathbb{X})$ is defined as 
\begingroup
\allowdisplaybreaks
\begin{equation}\label{eqn:robustbellman}
 (\tilde \Phi^{N,\epsilon} v)(x) \defeq \min_{a \in A(x)}\sup_{\nu \in {\cal P}^N(\epsilon)}\Bigg[\int_\mathbb{W} \left(c(x,a,z) + \alpha  v(F(x,a,z))\right)\nu(dz) \Bigg] \ \forall x \in \mathbb{X}.
\end{equation}
\endgroup
For any $t \in \mathbb{N}$, the operator $\tilde \Phi^{N,\epsilon}_t$ on $\Sigma_b(\mathbb{X})$ is defined as $\tilde \Phi^{N,\epsilon}_t v \defeq \tilde \Phi^{N,\epsilon} \tilde \Phi^{N,\epsilon}_{t-1}v$ where $\tilde\Phi^{N,\epsilon}_0 v \defeq  v$. If Assumptions \ref{assum:mdpassum} and \ref{assum:ambiguitycompact} hold, then\\
\noindent (a) The robust optimal value function $\tilde J^{N,\epsilon}$ is the unique fixed point of $\tilde \Phi^{N,\epsilon}$.\\ 
\noindent (b)   There exists a deterministic stationary policy that is robust optimal. A policy $\hat\pi^N \in \Pi$ is robust optimal if and only if it satisfies
\begingroup
\allowdisplaybreaks
\begin{equation}\label{eqn:rmdp_policy_bellman}
\tilde J^{N,\epsilon}(x) = \sup_{\nu \in {\cal P}^N(\epsilon)}\\ \int_\mathbb{W} (c(x,\hat\pi^N(x),z) + \alpha  \tilde J^{N,\epsilon}(F(x,\hat\pi^N(x),z)))\nu(dz) \ \forall x \in \mathbb{X}.
\end{equation}
\endgroup
\noindent (c) $\sup_{x \in \mathbb{X}}\big|(\tilde \Phi_t^{N,\epsilon} {\bf 0})(x)- \tilde J^{N,\epsilon}(x)\big| \leq \norm{c}_\infty \alpha^t/(1 - \alpha) \ \forall t \in \mathbb{N}$.
\end{thm}

The usage of ``min'' instead of ``inf'' in the outer optimization problem in the right-hand side of \eqref{eqn:robustbellman} is justified under Assumptions \ref{assum:mdpassum} and \ref{assum:ambiguitycompact}  \cite[Section 2.1]{minjarez2020zero}.

We provide remarks regarding  \Cref{thm:robustbellman}. Since the empirical distribution, $\hat\mu^N$, is a random object, the ambiguity set, ${\cal P}^N(\epsilon)$, is a random set, and hence the minimax problem \eqref{eqn:robustmdp} is a random optimization problem. We use realizations of $\hat\mu^N$ to instantiate ${\cal P}^N(\epsilon)$, and by extension, problem \eqref{eqn:robustmdp}. That is, every realization of $\hat\mu^N$ induces a realization of problem \eqref{eqn:robustmdp}, which is a (fixed) instance of an   RMDP. Under Assumptions \ref{assum:mdpassum} and \ref{assum:ambiguitycompact}, this instance admits an optimality characterization as outlined in \Cref{thm:robustbellman}. Since Assumptions \ref{assum:mdpassum} and \ref{assum:ambiguitycompact} do not place any restrictions on $\hat\mu^N$, \Cref{thm:robustbellman} holds for all realizations of \eqref{eqn:robustmdp}. We overload $\hat\mu^N, {\cal P}^N(\epsilon),  \tilde J^{N,\epsilon}, \hat\pi^N, \tilde \Phi_t^{N,\epsilon}$, etc to denote the random objects as well as their respective
realizations. The specific usage will be clear from the context. Since instances of  \eqref{eqn:robustmdp} are instantiated via realizations of $\hat\mu^N$, the distribution of the random objects is induced by the distribution of $\hat\mu^N$, which in turn is induced by the product measure $\mu_\otimes^N$ on $\mathbb{W}^N$.

A characterization of  adversary's best response
under conditions similar to Assumptions \ref{assum:mdpassum} and \ref{assum:ambiguitycompact} appears in \cite[Theorem 4.18]{glauner2020robust}. Our proofs do not use any structural properties of the adversary's best response to establish the stated claims. Moreover, we treat the data-driven RMDP as a tool to select a policy for the MDP instance with unknown disturbance distribution. Hence, we focus on  the performance of the robust optimal policy, $\hat\pi^N$, on the considered MDP instance.

The next theorem presents the optimality characterization for MDPs.

\begin{thm}
\normalfont
\cite[Chapter 4]{hernandez2012discrete}\label{thm:bellman}
Consider the MDP in \eqref{eqn:mdpoptimality}. For $v \in \Sigma_b(\mathbb{X})$ and $\pi \in \Pi$, define the Bellman operator $\Phi$ and the evaluation operator $\Phi_\pi$ on $\Sigma_b(\mathbb{X})$ as 
\begingroup
\allowdisplaybreaks
\begin{align}\label{eqn:bellman}
 (\Phi v)(x) &\defeq \min_{a \in A(x)}\left[\int_\mathbb{W} \left(c(x,a,z) + \alpha  v(F(x,a,z))\right)\mu(dz) \right] \ \forall x \in \mathbb{X}\\
 \label{eqn:eval}
 (\Phi_\pi v)(x) &\defeq \int_\mathbb{W} \left(c(x,\pi(x),z) + \alpha  v(F(x,\pi(x),z))\right)\mu(dz) \ \forall x \in \mathbb{X}.
\end{align}
\endgroup
For $t \in \mathbb{N}$, the operators $\Phi_t$ and $\Phi_{t,\pi}$ on $\Sigma_b(\mathbb{X})$ are defined as $\Phi_t v \defeq  \Phi \Phi_{t-1}v$ and $\Phi_{t,\pi}v \defeq \Phi_\pi\Phi_{t-1,\pi}v$, where $\Phi_0 v \defeq \Phi_{0,\pi}v \defeq v$. If \Cref{assum:mdpassum} holds, then\\
\noindent (a) The optimal value function $J^*$ is the unique fixed point of $\Phi$.\\
\noindent (b) There exists a deterministic stationary policy that is optimal for  \eqref{eqn:mdpoptimality}. Any $\pi^* \in \Pi$ is an optimal policy if and only if it satisfies 
\begingroup
\allowdisplaybreaks
\begin{equation}\label{eqn:policy_bellman}
J^*(x) = \int_\mathbb{W} (c(x,\pi^*(x),z) +  \alpha  J^*(F(x,\pi^*(x),z)))\mu(dz)  \ \forall x \in \mathbb{X}.
\end{equation}
\endgroup
\noindent (c) $ \sup_{x \in \mathbb{X}, \pi \in \Pi}\abs{(\Phi_{t,\pi} v)(x) - J(\pi,x)} \leq \alpha^t({\norm{c}_\infty}/(1-\alpha) + \norm{v}_\infty) \ \forall t \in \mathbb{N}
 $. 
\end{thm}

The usage of ``min'' instead of ``inf'' in the right-hand side of \eqref{eqn:bellman} is a consequence of \Cref{assum:mdpassum} imposed on the components of the MDP \cite[Theorem 3.3.5]{hernandez2012discrete}.

\section{Data-driven performance guarantees}\label{sec:data-driven_rmdp}
In \Cref{sec:valueconv}, we analyze the asymptotic properties of data-driven RMDPs as the sample-size $N \to \infty$. \Cref{sec:probguarantee} establishes the probabilistic guarantee on the out-of-sample value function. The convergence rate and sample complexity results are presented in \Cref{sec:conv_rate} and \Cref{sec:sample_complexity}, respectively. Finally, the out-of-distribution perfomance claims are analyzed in \Cref{sec:out_of_dist_bounds}. 

The data-driven performance guarantees established in this section also hold for nonstationary finite-horizon problems. We skip stating and deriving them for brevity.

\subsection{Value convergence to true optimal value}\label{sec:valueconv}
The asymptotic convergence results rely on the following assumption on the distance function.

\begin{assum}\label{assum:distancemetrize}
 \normalfont
 The distance function satisfies the property that $\lim_{k \to \infty}d(\nu^k,\rho^k) = 0$ implies $\lim_{k \to \infty}\beta(\nu^k,\rho^k) = 0$ for all sequences $\{\nu^k\}, \{\rho^k\} \in {\cal M}(\mathbb{W})$.
\end{assum}

\Cref{assum:distancemetrize} postulates that convergence of probability distributions with respect to the distance function ensures convergence with respect to the bounded Lipschitz metric. We present the motivation behind \Cref{assum:distancemetrize}. Since the empirical distribution, $\hat \mu^N$, varies with the sample-size, $N$, the ambiguity sets, ${\cal P}^N(\epsilon)$, also keep moving with $N$. Moreover, $\hat \mu^N$ weakly converges to $\mu$ \cite[Theorem 11.4.1]{dudley2018real}. Hence, it seems like ${\cal P}^N(\epsilon)$ converges (in an appropriate sense) to $\mu$ if $\epsilon$ goes to $0$ with increasing $N$. This reasoning would be true if distributions that are close to one another with respect to the distance function are also close in the topology of weak convergence, which is the statement of \Cref{assum:distancemetrize}. Hence, if $\epsilon \to 0$ as $N \to \infty$ and the distance function satisfies  \Cref{assum:distancemetrize}, we can expect ${\cal P}^N(\epsilon)$ to move and shrink simultaneously, and eventually converge to $\mu$. This intuition is formalized in \Cref{lem:ballconv}.  Since direct verification of \Cref{assum:distancemetrize} is challenging, the following assumption serves as a sufficient condition for \Cref{assum:distancemetrize}.
\begin{assum}\label{assum:betametrize}
\normalfont
There exists a continuous function $\psi:\mathbb{R}_+ \to \mathbb{R}_+$ with $\psi(0) = 0$ such that $\beta(\nu_1,\nu_2) \leq \psi(d(\nu_1,\nu_2))$ for all $\nu_1,\nu_2\in {\cal M}(\mathbb{W})$.
\end{assum}

\begin{lem}\label{lem:betametrize}
 \normalfont
 If \Cref{assum:betametrize} holds, then \Cref{assum:distancemetrize} holds.
\end{lem}
The proof of \Cref{lem:betametrize} follows via similar proof arguments of \cite[Lemma 1]{ramani2022robust}, and hence skipped. Assumptions \ref{assum:distancemetrize} and \ref{assum:betametrize} are adaptations of the ``metrizing'' property of the distance function presented in \cite[Assumption 1, 2]{ramani2022robust} for probability distributions defined on finite sample spaces. \Cref{assum:betametrize} is satisfied by many well-known distance functions studied in the literature, including TV, KL, Hellinger, $\chi^2$, Wasserstein, bounded Lipschitz metric, and Prokhorov metric. For these distances, the formula for the corresponding function $\psi(\cdot)$ is presented in  \Cref{sec:controldist}.

We now state the robust value function convergence result. For the sake of presentation, the proofs of the results in this paper along with the supporting claims are presented in the appendices.

\begin{thm}\label{thm:robustconvergence}
\normalfont
Suppose the radii of the ambiguity sets depend on the sample-size $N$ such that $\lim_{N \to \infty}\epsilon^N = 0$. If Assumptions \ref{assum:mdpassum}, \ref{assum:ambiguitycompact}, and \ref{assum:distancemetrize} hold, then $\mu_\otimes^\infty \big[\lim_{N \to \infty}\tilde J^{N,\epsilon^N}(x) = J^*(x) \ \forall x \in \mathbb{X}\big] = 1$.
\end{thm}

\begin{cor}\label{cor:robustintegralconv}
\normalfont
Suppose the radii of the ambiguity sets depend on the sample-size $N$ such that $\lim_{N \to \infty}\epsilon^N = 0$. Let $\rho \in {\cal M}(\mathbb{X})$ be an initial state distribution over $\mathbb{X}$. If Assumptions \ref{assum:mdpassum}, \ref{assum:ambiguitycompact}, and \ref{assum:distancemetrize} hold, then $\mu_\otimes^\infty\big[\lim_{N \to \infty}\int_{\mathbb{X}}\tilde J^{N,\epsilon^N}(x) \rho(dx) = \int_\mathbb{X} J^*(x) \rho(dx)\big] = 1$. 
\end{cor}

Recall from \Cref{sec:introduction}, the out-of-sample value function, $J(\hat\pi^N,x) \ \forall x \in \mathbb{X}$, is the value of the robust optimal policy $\hat\pi^N$ in the MDP with true disturbance distribution, i.e., $\mu$. The next result establishes the convergence of $J(\hat\pi^N,x)$ to $J^*(x)$, as $N \to \infty$.

\begin{thm}\label{thm:oosconvergence}
\normalfont
Suppose the radii of the ambiguity sets depend on the sample-size $N$ such that $\lim_{N\rightarrow \infty}\epsilon^N = 0$. If Assumptions \ref{assum:mdpassum}, \ref{assum:ambiguitycompact}, and \ref{assum:distancemetrize} hold, then $\mu_\otimes^\infty\big[\lim_{N \to \infty}J(\hat\pi^N, x) = J^*(x) \ \forall x \in \mathbb{X}\big] = 1$.
\end{thm}

\begin{cor}\label{cor:oosintegralconv}
\normalfont
Suppose the radii of the ambiguity sets depend on the sample-size $N$ such that $\lim_{N \to \infty}\epsilon^N = 0$. Let $\rho \in {\cal M}(\mathbb{X})$ be an initial state distribution over $\mathbb{X}$. If Assumptions \ref{assum:mdpassum}, \ref{assum:ambiguitycompact}, and \ref{assum:distancemetrize} hold, then $\mu_\otimes^\infty\big[\lim_{N \to \infty}\int_{\mathbb{X}} J(\hat\pi^N,x) \rho(dx) = \int_\mathbb{X} J^*(x) \rho(dx)\big] = 1$. 
\end{cor}

\subsection{Probabilistic guarantee on the performance of robust optimal policy}\label{sec:probguarantee}
Though \cref{thm:oosconvergence} establishes the convergence of the out-of-sample value function to the true optimal value function, a decision-maker may also want to compute the out-of-sample value function for finite sample-sizes. However, the out-of-sample value function cannot be determined in practice since it requires the knowledge of the unknown true distribution. Hence, we settle for determining an upper bound on the out-of-sample value function that holds with a high probability for  finite sample-sizes. In this section, we demonstrate how to design the data-driven ambiguity sets so that the robust optimal value function serves as a high probability upper bound for the out-of-sample value function.

\begingroup
\allowdisplaybreaks
\begin{prop}\label{prop:prob_perf_pointwise}
 \normalfont
 Fix a sample-size $N\geq 1$ and $\epsilon \in [0,\infty)$. Suppose Assumptions \ref{assum:mdpassum} and  \ref{assum:ambiguitycompact}  hold. If $d(\mu,\hat\mu^N) \leq \epsilon$ (i.e., $\mu \in {\cal P}^N(\epsilon)$), then $J(\hat\pi^N,x) \leq \tilde J^{N,\epsilon}(x) \ \forall x \in \mathbb{X}$.
\end{prop}
\endgroup

\Cref{prop:prob_perf_pointwise} states that if a realization of $\hat\mu^N$ satisfies $d(\mu,\hat\mu^N) \leq \epsilon$, then the robust optimal value function is an upper bound on the out-of-sample value function under the corresponding realization. Ensuring $d(\mu,\hat\mu^N) \leq \epsilon$ holds with high probability then establishes the probabilistic performance guarantee result. To achieve this, we rely upon \Cref{assum:controlconc} presented below. Roughly speaking, the assumption states that it is possible to tune the radii of the data-driven ambiguity sets so that those ambiguity sets contain the true distribution corresponding to any desired confidence level. This assumption was also used in \cite{ramani2022robust,ramani2024family} to establish the probabilistic performance guarantee result for data-driven RMDPs over finite state- and action-spaces.

\begin{assum}\label{assum:controlconc} 
\normalfont
Fix a sample-size $N \geq 1$, a probability distribution $\nu \in {\cal M}(\mathbb{W})$, and $\gamma \in (0,1)$. Let $\hat \nu^N$ be the empirical distribution of $\nu$ constructed from $N$ iid samples. The distance function satisfies the following property: there exists an $0<\epsilon^N_\gamma<\sup_{\nu_1,\nu_2\in {\cal M}(\mathbb{W})}\ d(\nu_1,\nu_2)$ such that
\begingroup
\allowdisplaybreaks
\begin{equation}
\label{eqn:concineq}
\nu_\otimes^N\big[d(\nu,\hat \nu^N)\leq \epsilon^N_\gamma\big] \geq 1-\gamma.
\end{equation}
\endgroup
\end{assum}

The parameter $\epsilon^N_\gamma$ generally depends on the quantities associated with properties of $\mathbb{W}$ and $\nu$, but they are omitted from the notation for brevity. Also, additional assumptions on $\mathbb{W}$ and $\nu$ might be required for the distance function to satisfy \Cref{assum:controlconc}. In \Cref{sec:controldist}, we present distances studied in the literature that satisfy \Cref{assum:controlconc} for concrete choices of $\mathbb{W}$ and $\nu$. For those distances, the corresponding closed-form expressions of $\epsilon^N_\gamma$ are presented in that section. Those formulas satisfy the additional property that $\lim_{N \to \infty}\epsilon^N_\gamma = 0$. Hence, for those distances, setting the  ambiguity set radius to $\epsilon^N_\gamma$ ensures that the radii conditions of Theorems \ref{thm:robustconvergence} and \ref{thm:oosconvergence} hold.

The following theorem states that  if there exists a $\epsilon^N_\gamma$ for a distance that is consistent with \eqref{eqn:concineq}, then setting the ambiguity set radius to $\epsilon^N_\gamma$  ensures that the robust optimal value function serves as a high probability upper bound for the out-of-sample value function.

\begin{thm}\label{thm:controlinfiniteprob} 
\normalfont 
Fix a sample-size $N\geq 1$ and $\gamma\in (0,1)$. Suppose Assumptions \ref{assum:mdpassum}, \ref{assum:ambiguitycompact}, and  \ref{assum:controlconc} hold, and let 
$0<\epsilon^N_{\gamma} < \sup_{\nu_1,\nu_2 \in {\cal M}(\mathbb{W})}\ d(\nu_1,\nu_2)$ be as in \cref{assum:controlconc}. 
Then
\begingroup
\allowdisplaybreaks
\begin{equation}
\label{eqn:infiniteprob}
\mu_\otimes^N \big[J(\hat\pi^N,x) \leq \tilde J^{N,\epsilon^N_\gamma}(x) \ \forall x \in \mathbb{X} \big]\geq 1-\gamma.
\end{equation}
\endgroup
\end{thm}

To illustrate the form of  \Cref{thm:controlinfiniteprob} when used with a specific distance, let $\mathbb{W} \subseteq \mathbb{R}^m$ for some arbitrary positive integer $m$, and suppose the true distribution satisfies the light-tailed condition $\int_{\mathbb{W}} \exp(\lambda \norm{z}^a) \mu(dz) < \infty$ for some $\lambda > 0, \ a > 1$. This condition is satisfied by  subgaussian distributions;  see  \Cref{sec:wasserstein}. Fix a sample-size $N \geq 1$ and a  confidence level $1-\gamma \in (0,1)$. Suppose that the ambiguity set ${\cal P}^N(\epsilon^N_\gamma)$ is constructed using the Wasserstein distance with $\epsilon^N_\gamma=\epsilon^N_\gamma(\text{Wasserstein})$, where $\epsilon^N_\gamma(\text{Wasserstein})$ is defined in \eqref{eqn:wassestein_radius} (which is compatible with \eqref{eqn:concineq}). Then, with probability at least $1-\gamma$, the out-of-sample value $J(\hat\pi^N,x)$ is at most the corresponding robust optimal value $\tilde J^{N,\epsilon^N_\gamma}(x)$ simultaneously for all $x \in \mathbb{X}$. A formula for $\epsilon^N_\gamma$ that is consistent with \eqref{eqn:concineq} is presented for the bounded Lipschitz metric and the Prokhorov metric in  subsections \ref{sec:beta} and \ref{sec:prokhorov}, respectively.

\subsection{Probabilistic convergence rate}\label{sec:conv_rate}
This section derives the rate at which the robust optimal value function and the out-of-sample value function converge to the true optimal value function in terms of sample-size and ambiguity set radius. In this section, we assume that the components of the MDP satisfy the following conditions.

\begin{assum}\label{assum:conv_rate_assum}
\normalfont
In addition to \Cref{assum:mdpassum}, the MDP model satisfies the following conditions.
\begin{enumerate}[(a)]
 \item $A(x) = \mathbb{A}$ for all $x \in \mathbb{X}$ and $\mathbb{A}$ is compact. Hence, $\mathbb{K} = \mathbb{X} \times \mathbb{A}$.
 \item There exists $L_c, L_F \in [0,\infty)$ such that for all $(x,a) \in \mathbb{K}$ and $(a,z) \in \mathbb{A} \times \mathbb{W}$, 
 \begin{enumerate}[(i)]
 \item $c(\cdot,a,z)$ and $c(x,a,\cdot)$ are $L_c$-Lipschitz continuous.
\item $F(\cdot,a,z)$ and $F(x,a,\cdot)$ are $L_F$-Lipschitz continuous.
\end{enumerate}
\item $\alpha L_F < 1$.
\end{enumerate}
\end{assum}

The next result establishes a  probabilistic convergence rate of the robust optimal and out-of-sample value functions in terms of the sample-size and ambiguity set radii.

\begin{thm}\label{thm:conv_rate}
\normalfont
Fix a sample-size $N \geq 1$ and $\epsilon \in [0,\infty)$. Suppose \Cref{assum:conv_rate_assum} holds. Let the distance function $d$ satisfy \Cref{assum:betametrize}, and the function $\psi$ introduced there is nondecreasing. If there exists $\eta: \mathbb{N} \times [0,\infty) \to [0,1]$ such that
\begingroup
\allowdisplaybreaks
\begin{equation}\label{eqn:concineq_conv}
\mu_\otimes^N\big[d(\mu,\hat \mu^{\bar N}) \leq \bar\epsilon\big] \geq 1 - \eta(\bar N,\bar\epsilon) \ \forall \bar N \in \mathbb{N}, \,  \bar\epsilon \in [0,\infty),
\end{equation}
\endgroup
then
\begingroup
\allowdisplaybreaks
\begin{align}
\label{eqn:oos_rate}
\mu_\otimes^N\Big[J(\hat\pi^N,x) - J^*(x) \leq \tilde J^{N,\epsilon}(x) - J^*(x) \leq 2\, \psi(\epsilon){\bf \Delta} \ \forall x \in \mathbb{X} \Big] \geq 1 - \eta(N,\epsilon),
\end{align}
\endgroup
where ${\bf \Delta}  \defeq \norm{c}_\infty/(1-\alpha)^2 + L_c/(1-\alpha) + \alpha L_c L_F/((1-\alpha)(1-\alpha L_F))$.
\end{thm}

\Cref{thm:conv_rate}, in addition to relying on \Cref{assum:conv_rate_assum}, also required the distance function to satisfy \Cref{assum:betametrize} and further assumed that the $\psi$ function introduced in \Cref{assum:betametrize} is nondecreasing. We note that the nondecreasing property of $\psi$ is mild, and, like \Cref{assum:betametrize}, it is satisfied by many well-known distance functions, including all the distances listed in \Cref{sec:controldist}. 

We provide an interpretation of \eqref{eqn:oos_rate}.  Since the $\psi$ function in \Cref{thm:conv_rate} is nondecreasing, for any $x \in \mathbb{X}$, the error bound $J(\hat\pi^N,x) - J^*(x)$ in  \eqref{eqn:oos_rate} can be made arbitrarily small by picking sufficiently small values for the ambiguity set radius, $\epsilon$. 
Next, from \eqref{eqn:concineq_conv}, we note that any meaningful estimate of $\eta(N,\epsilon)$ will be nonincreasing in $\epsilon$ since $\{d(\mu,\hat\mu^N) \leq \epsilon_1\} \subseteq \{d(\mu,\hat\mu^N) \leq \epsilon_2\}$ for all $0 \leq \epsilon_1 < \epsilon_2 < \infty$. This implies that the confidence level, $1-\eta(N,\epsilon)$, claimed in \eqref{eqn:oos_rate} is nondecreasing in $\epsilon$. Therefore, for any fixed sample-size, the probabilistic convergence rate result in  \eqref{eqn:oos_rate} highlights the tradeoff between the error bound and its associated confidence level as a function of ambiguity set radius. A natural extension of \Cref{thm:conv_rate} is to check,  for any specified $\delta > 0$ and $\gamma \in (0,1)$, if it is possible to tune the sample-size and the ambiguity set radius  to make the error bound smaller than $\delta$ with a probability of at least $1-\gamma$? This is answered in the affirmative in \Cref{sec:sample_complexity}, where we appropriately reverse-engineer the convergence rate result presented in \Cref{thm:conv_rate} to obtain concrete values for the sample-size and the ambiguity set radius.   

\subsection{Sample complexity}
\label{sec:sample_complexity}
Given $\delta > 0$ and $\gamma \in (0,1)$, the sample complexity is the minimum number of samples (say, $N^*$) of the disturbance variable needed to ensure that  $J(\hat\pi^N,x)-J^*(x) \leq \delta$ holds simultaneously for all $x \in \mathbb{X}$ with probability at least $1-\gamma$ for sample-sizes $N \geq N^*$. In this section, we derive a lower bound on $N^*$ utilizing the convergence rate established in \Cref{thm:conv_rate}. For ease of illustration, we derive the sample complexity using the Wasserstein distance, though the technique is applicable to any distance function that satisfies condition \eqref{eqn:concineq_conv} presented in \Cref{thm:conv_rate}.

Let $\mathbb{W} \subseteq  \mathbb{R}^m$ and $\mu$ satisfy $\int_\mathbb{W} \exp(\lambda\norm{z}^a) \mu(dz) < \infty$ for some $a > 1, \lambda > 0$. For concreteness, let $m \geq 3$ and $a \geq m$. From \eqref{eqn:wasserstein_concentration}, $\mu_\otimes^N[d_{\text{Wasserstein}}(\mu,\hat\mu^N) \leq \epsilon] \geq 1 - c_1 \exp(-c_2N\epsilon^m)1_{\{\epsilon \leq 1\}} - c_1\exp(-c_2 N\epsilon^a)1_{\{\epsilon > 1\}} \geq 1 - 2c_1 \exp(-c_2N\epsilon^m)$, where the last inequality follows since $a \geq m$. Here, $c_1,c_2$ are positive constants that depend only on $m, a, \int_\mathbb{W} \exp(\lambda \norm{z}^a)\mu(dz)$. Hence, \eqref{eqn:concineq_conv}   is satisfied with $\eta(N,\epsilon) = 2c_1 \exp(-c_2N\epsilon^m)$. From  \eqref{eqn:oos_rate}, $\mu_\otimes^N[J(\hat\pi^N,x) - J^*(x) \leq \tilde J^{N,\epsilon}(x) - J^*(x) \leq 2\psi(\epsilon){\bf \Delta} \ \forall x \in \mathbb{X} \ \forall x \in \mathbb{X}] \geq 1-2c_1\exp(-c_2N\epsilon^m)$. Using the above result, we compute $\epsilon_{\text{ub}}$ such that $\mu_\otimes^N[J(\hat\pi^N,x)-J^*(x) \leq \tilde J^{N,\epsilon}(x) - J^*(x) \leq \delta \ \forall x \in \mathbb{X}] \geq 1-2c_1\exp(-c_2N\epsilon^m) \, \forall \epsilon \leq \epsilon_{\text{ub}}$. Next, we compute $\epsilon_{\text{lb}}$ such that $\mu_\otimes^N[J(\hat\pi^N,x) - J^*(x) \leq \tilde J^{N,\epsilon}(x) - J^*(x) \leq 2\psi(\epsilon){\bf \Delta} \ \forall x \in \mathbb{X}] \geq 1-\gamma \, \forall \epsilon \geq \epsilon_{\text{lb}}$. Finally, we equate $\epsilon_{\text{lb}}$ to $\epsilon_{\text{ub}}$ and determine the sample complexity.

Note that $J(\hat\pi^N,x) - J^*(x) \leq \tilde J^{N,\epsilon}(x) - J^*(x) \leq \delta$ holds simultaneously for all $x \in \mathbb{X}$ if $2 \psi(\epsilon){\bf \Delta}\leq \delta$. For the Wasserstein distance, we prove in \Cref{sec:wasserstein} that $\psi(\epsilon) = \epsilon$ for all $\epsilon \in \mathbb{R}_+$. Hence, substituting $\psi(\epsilon) = \epsilon$ in the expression $2 \psi(\epsilon){\bf \Delta} \leq \delta$ and solving for $\epsilon$ gives an upper bound on $\epsilon$ as
\begingroup
\allowdisplaybreaks
\begin{equation}
\label{eqn:radius_upper_bound}
 \epsilon_{\text{ub}} \defeq \frac{\delta}{2{\bf \Delta}} = \frac{1}{2}\left(\frac{\delta(1-\alpha)^2(1-\alpha L_F)}{\norm{c}_\infty(1-\alpha L_F) + L_c(1-\alpha)}\right),
\end{equation}
\endgroup
where the last equality follows by plugging the definition of ${\bf \Delta}$ from \Cref{thm:conv_rate} followed by algebraic calculation. 
Next,  $\mu_\otimes^N[J(\hat\pi^N,x) - J^*(x) \leq \tilde J^{N,\epsilon}(x) - J^*(x) \leq 2\psi(\epsilon){\bf \Delta} \ \forall x \in \mathbb{X}] \geq 1-\gamma$ if $1-\eta(N,\epsilon) \geq 1-\gamma$. This leads  to the condition $1-2c_1\exp(-c_2N\epsilon^m) \geq 1- \gamma$ after substituting for the value of $\eta(N,\epsilon)$. Solving for $\epsilon$ in  $1-2c_1\exp(-c_2N\epsilon^m) \geq 1- \gamma$ provides a lower bound on $\epsilon$ as
\begingroup
\allowdisplaybreaks
\begin{equation}
 \label{eqn:radius_lower_bound}
 \epsilon_{\text{lb}} \defeq \left(\frac{\log(2c_1/\gamma)}{Nc_2}\right)^{1/m}.
\end{equation}
\endgroup
Hence, $\mu_\otimes^N[J(\hat\pi^N,x) - J^*(x) \leq \tilde J^{N,\epsilon}(x) - J^*(x) \leq \delta \ \forall x \in \mathbb{X}] \geq 1-\gamma$ if $\epsilon_{\text{lb}} \leq \epsilon \leq \epsilon_{\text{ub}}$, i.e.,
\begingroup
\allowdisplaybreaks
\begin{equation}
 \label{eqn:radius_lower_upper_bound}
 \left(\frac{\log(2c_1/\gamma)}{Nc_2}\right)^{1/m} \leq \epsilon \leq \frac{1}{2}\left(\frac{\delta(1-\alpha)^2(1-\alpha L_F)}{\norm{c}_\infty(1-\alpha L_F) + L_c(1-\alpha)}\right).
\end{equation}
\endgroup
Setting the value of the leftmost expression in \eqref{eqn:radius_lower_upper_bound} to be at most the value of the  corresponding rightmost expression and then solving for $N$ gives us
\begingroup
\allowdisplaybreaks
\begin{equation}
 \label{eqn:sample_complexity}
 N \geq \frac{2^m \log(2c_1/\gamma)\Big((\norm{c}_\infty)(1-\alpha L_F) + L_c(1-\alpha)\Big)^m}{c_2\delta^m(1-\alpha)^{2m}(1-\alpha L_F)^{m}}.
\end{equation}
\endgroup

The above derivation of sample complexity has the following implications. Fix $\delta > 0$ and $\gamma \in (0,1)$. Let the number of samples, $N$, collected by the decision-maker satisfy \eqref{eqn:sample_complexity}. Suppose the decision-maker uses this value of $N$ to calculate the leftmost expression in \eqref{eqn:radius_lower_upper_bound}, and then picks the ambiguity set radii in accordance with \eqref{eqn:radius_lower_upper_bound}. Then $\mu_\otimes^N[J(\hat\pi^N,x) - J^*(x) \leq \tilde J^{N,\epsilon}(x) - J^*(x) \leq \delta \ \forall x \in \mathbb{X}] \geq 1 - \gamma$, i.e., with probability at least $1-\gamma$, the robust optimal policy $\hat\pi^N$ is $\delta$-optimal simultaneously for all $x \in \mathbb{X}$. Since $J(\hat\pi^N,x) - J^*(x) \leq \tilde J^{N,\epsilon}(x) - J^*(x) \leq \delta$ implies $J(\hat\pi^N,x) - J^*(x) \leq \delta$ and $J(\hat\pi^N,x) \leq \tilde J^{N,\epsilon}(x)$, we have $\mu_\otimes^N[J(\hat\pi^N,x)-J^*(x) \leq \delta, \ J(\hat\pi^N,x) \leq \tilde J^{N,\epsilon}(x) \ \forall x \in \mathbb{X}] \geq 1-\gamma$. That is,  for any specified confidence level, the suboptimality gap of the robust optimal policy can be made arbitrarily small while simultaneously guaranteeing that  the robust optimal value function is an upper bound on the out-of-sample value function.
\begin{rem}
\normalfont
 We note that the compactness of $\mathbb{A}$ in \Cref{assum:conv_rate_assum}(a) only serves as a sufficient condition for existence of Bellman (resp. robust Bellman) equation and related optimality characterizations. More specifically, the compactness of $\mathbb{A}$ is never used to establish the convergence rate result in \Cref{thm:conv_rate} or derive the sample complexity bound.  
\end{rem}

\subsection{Out-of-distribution performance bounds}\label{sec:out_of_dist_bounds}
In this section, we analyze the suboptimality gap incurred due to deploying a robust optimal policy in a MDP whose disturbance distribution is not the same as the distribution from which iid samples are drawn. The formal setup is as follows. Consider ${\cal M}^{\text{true}} \defeq \langle\mathbb{X},\mathbb{A},\mathbb{W},F, \mu^{\text{true}},c,\alpha\rangle$ with unknown $\mu^{\text{true}}$. Further, sampling from $\mu^{\text{true}}$ is challenging. Hence, we use a proxy MDP whose model components are same as ${\cal M}^{\text{true}}$, except for the distribution of the disturbance variable, denoted, $\mu$ ($\neq \mu^{\text{true}}$). Like $\mu^{\text{true}}, \mu$ is also unknown, but samples from $\mu$ can be obtained. Hence, we obtain the empirical estimate $\hat \mu^N$ of $\mu$ using $N$ iid samples drawn from  $\mu$, construct the ambiguity set ${\cal P}^N(\epsilon)$, and then solve the resulting data-driven RMDP to obtain the robust optimal value function $\tilde J^{N,\epsilon}$ and  a robust optimal policy $\hat \pi^N$. The goal is to analyze the suboptimality gap incurred due to  deploying $\hat \pi^N$ in ${\cal M}^{\text{true}}$. \Cref{thm:ood_rate} establishes a bound on this suboptimality gap, referred the out-of-distribution bound.  For ${\cal M}^{\text{true}}$, we superscript its corresponding entities using the label ``true'', i.e., $\Phi_{t}^{\text{true}}, J^{*,\text{true}}, J^{\text{true}}(\hat\pi^N,\cdot)$, etc.

\begin{thm}\label{thm:ood_rate}
\normalfont
Fix $N \geq 1, \epsilon \in [0,\infty)$. Let the components of the true and proxy MDP satisfy \Cref{assum:conv_rate_assum}. Let $d$ satisfy \Cref{assum:betametrize}, and the function $\psi$ stated there is nondecreasing. Also, let  $d$ satisfy  \eqref{eqn:concineq_conv}, i.e., $\exists \ \eta(N,\epsilon) \in [0,1]$ such that $\mu_\otimes^N\big[d(\mu,\hat \mu^N) \leq \epsilon\big] \geq 1 - \eta(N,\epsilon)$. Let ${\bf \Delta}$ be as defined in \Cref{thm:conv_rate}. Then
\[
\mu_\otimes^N\Big[J^{\text{true}}(\hat\pi^N,x) - J^{*,\text{true}}(x) \leq \big(2\, \psi(\epsilon)+ \beta(\mu^{\text{true}},\mu)\big)\frac{{\bf \Delta}}{1-\alpha} \ \forall x \in \mathbb{X} \Big] \geq 1 - \eta(N,\epsilon).
\]
\end{thm}

We present an interpretation of the claim in \Cref{thm:ood_rate}. The upper bound on $J^{\text{true}}(\hat\pi^N,x) - J^{*,\text{true}}(x)$ is the sum of two components,  
 $2\, \psi(\epsilon){\bf \Delta}/(1-\alpha)$ and $\beta(\mu^{\text{true}},\mu){\bf \Delta}/(1-\alpha)$. The term  $2\, \psi(\epsilon){\bf \Delta}/(1-\alpha)$ depends on the ambiguity set radius and can be made small enough to any desired level by varying the radius. For instance, by allowing $\epsilon \to 0$ with increasing $N$, this term vanishes as $N \to \infty$. For this reason, we call this term the statistical error.  The second component $\beta(\mu^{\text{true}},\mu){\bf \Delta}/(1-\alpha)$, referred the nonstatistical error, captures the discrepancy between the true and proxy MDP. It is  independent of $\epsilon$ and $N$, and it equals $0$ if and only if $\mu = \mu^{\text{true}}$.
 
 In general, analyzing the  out-of-distribution bounds is tightly coupled with the decision-making setup  under consideration. A recent paper \cite{wagenmaker2024overcoming} establishes out-of-distribution bounds for finite-horizon MDPs on finite spaces. Motivated by robotics applications, their decision-making scenario allows finetuning the policy obtained using proxy MDP based on an initial exploration on the true MDP. This extra exploration along with  assumptions on transition probabilities enabled \cite{wagenmaker2024overcoming} to design a tailor-made algorithm that reduces the  error to a value less than the distance between $\mu$ and $\mu^{\text{true}}$. Additionaly, \cite{wagenmaker2024overcoming} demonstrated that reducing the error to less than the distance between $\mu$ and $\mu^{\text{true}}$ is generally impossible without additional assumptions. A detailed study of the out-of-distribution bound of the robust optimal policy for various decision-making setups is outside the scope of this paper and deferred to future work.

\section{Empirical MDPs versus data-driven RMDPs}\label{sec:empirical_mdps}
In empirical MDPs, the unknown distribution of the random disturbance is replaced with its empirical distribution constructed from $N \in \mathbb{N}$ iid samples, and the resulting MDP is solved. We use $\tilde J^{N,*}_{\text{Emp}}$ to denote the optimal value function of the empirical MDP, called the empirical optimal value function. The corresponding optimal policy will be denoted $\hat\pi^N_{\text{Emp}}$, referred the empirical optimal policy.

\begin{thm}\label{thm:empirical_mdps_counterexample}
 \normalfont
 There exists a MDP instance with finite $\mathbb{X}, \mathbb{A}, \mathbb{W}$, and a state $x \in \mathbb{X}$ such that
 \begin{enumerate}[(a)]
 \item For all $\delta < 1$ and $\gamma  < 0.25$, at least $3$ samples are needed so that $J(\hat\pi^N_\text{Emp},x) - J^*(x) \leq \delta$ holds with probability at least $1 - \gamma$.
 \item $\mu_\otimes^N \left[J(\hat\pi^N_\text{Emp},x) \leq \tilde J^{N,*}_{\text{Emp}}(x) \right]
  \begin{cases}
   = 0.5 \text{ if the sample-size } N \text{ is odd. }\\
   = 0.75 \text{ if the sample-size } N = 2.\\
   \text{is decreasing over even sample-sizes and lies in }\\
   (0.5,0.75) \text{ for all even sample-sizes } N \neq 2.
   \end{cases}
    $
    \end{enumerate}
\end{thm}

\Cref{thm:empirical_mdps_counterexample} has the following implication. For the MDP instance constructed in the proof of that theorem, part (a) states that at least $3$ samples are required to bound the suboptimality gap of the empirical optimal policy to less than $1$ with a probability greater than $0.75$. Part (b) of the theorem implies that for every finite sample-size and any confidence level greater than $0.75$, the empirical optimal value function fails to serve as a lower bound on the out-of-sample value of the empirical optimal policy. Also, part (b) states that using a sample-size of $2$ maximizes the probability that the out-of-sample value of the empirical optimal policy is at most the corresponding empirical optimal value function. Together, parts (a) and (b) of the theorem provide the following stronger implication. For any $\delta < 1, \gamma < 0.25$, there exists no empirical optimal policy (corresponding to any sample-size $N$) that bounds the suboptimality gap to less than $\delta$ with a confidence level of $1-\gamma$ and maximize the probability that its out-of-sample value is at most the empirical optimal value function. This conclusion is in contrast to our data-driven RMDPs, where for any given confidence level, the robust optimal policy (corresponding to all sample-sizes beyond a threshold) bounds the suboptimality gap to any desired accuracy and simultaneously ensures that its out-of-sample value function is upper bounded by the robust optimal value function (see the discussion in the last paragraph of \Cref{sec:sample_complexity}).

The idea of evaluating  empirical optimization by comparing their in-sample versus out-of-sample performance was used in the context of single-stage stochastic optimization problems in \cite{bennouna2023certified,van2020data}. The authors in \cite{bennouna2023certified} prove that for every fixed sample-size, the expectation (taken with respect to the uncertainty in the sample data) of the empirical optimal value is less than the expectation of the corresponding out-of-sample value.  This result can be seen as an analog of part (b) of \Cref{thm:empirical_mdps_counterexample} which compares the in-sample and the out-of-sample performance in probability rather than in expectation. In \cite{van2020data}, the authors demonstrate via a counterexample that for all confidence levels beyond a threshold, the empirical optimal value fails to serve as an upper bound on the corresponding  out-of-sample value as the sample-size tends to infinity.  \Cref{thm:empirical_mdps_counterexample} extends the in-sample versus out-of-sample performance comparisons made in \cite{bennouna2023certified,van2020data} along multiple directions for empirical MDPs. It establishes probabilistic comparisons of in-sample versus out-of-sample performance for all finite sample-sizes and presents the tradeoff between sample complexity and probabilistic performance guarantee on the out-of-sample value function.

\section{Distances that satisfy Assumptions \ref{assum:ambiguitycompact},  \ref{assum:distancemetrize}, and \ref{assum:controlconc}}
\label{sec:controldist}
We begin with a sufficient condition on the distance function that guarantees the $\sigma$-weak compactness of ambiguity sets stated in  \Cref{assum:ambiguitycompact}.

\begin{lem}\label{lem:ballcompactness}
\normalfont
Suppose $\mathbb{W}$ is compact. Suppose the distance function $d: {\cal M}(\mathbb{W}) \times {\cal M}(\mathbb{W})\to [0,\infty)$ satisfies $d(\nu,\hat\nu)$ is weakly lower semicontinuous in $\nu$ for every fixed $\hat\nu$, i.e., $d(\nu,\hat\nu) \leq \liminf_{k \to \infty}d(\nu^k,\hat\nu)$ for all $\nu^k \to \nu$ weakly and $\hat\nu \in {\cal M}(\mathbb{W})$. Then  \Cref{assum:ambiguitycompact} holds.
\end{lem}
\begin{proof}
The compactness of $\mathbb{W}$ ensures that ${\cal M}(\mathbb{W})$ is weakly compact \cite[Proposition 7.22]{bertsekas1978stochastic}.  \Cref{assum:ambiguitycompact} then holds since for any $\hat\nu \in {\cal  M}(\mathbb{W})$ and $\epsilon \in [0,\infty)$, the weak lower semicontinuity of $d(\cdot,\hat\nu)$ implies $\{\nu \in {\cal M}(\mathbb{W}): d(\nu,\hat\nu) \leq \epsilon\}$ is weakly closed.
\end{proof}

The weak lower semicontinuity of $d(\cdot, \hat\nu)$ is a mild condition and satisfied by many  well-known distances. For instance, $d(\nu,\hat\nu)$ is weakly lower semicontinuous in $(\nu,\hat\nu)$ for $f$-divergences \cite[Theorem 1.47]{liese1987convex} and the Wasserstein distance \cite[Remark 6.12]{villani2009optimal}. For the bounded Lipschitz metric and the Prokhorov metric, $d(\nu,\hat\nu)$ is weakly continuous in $(\nu,\hat\nu)$ since these metrics metrize the topology of weak convergence \cite[Theorem 11.3.3]{dudley2018real}. We note that the weak compactness of ambiguity sets also holds under conditions on $\mathbb{W}$ and $d$ that are different from those in \Cref{lem:ballcompactness}; for example, the Wasserstein ambiguity set defined over proper Polish spaces \cite[Theorem 1]{yue2021linear}.

In the rest of the section, we list well-known distances that satisfy \Cref{assum:betametrize} along with their corresponding $\psi$ function defined in that assumption. Recall from    \Cref{lem:betametrize} that \Cref{assum:betametrize} is a sufficient condition for \Cref{assum:distancemetrize}. For the distances in our list that satisfy \Cref{assum:betametrize}, we present a concrete instance of $\mathbb{W}$ and $\nu$ under which \Cref{assum:controlconc} holds and state the corresponding closed-form expression for the ambiguity set radius, $\epsilon^N_\gamma$. 

\subsection{Total Variation distance}\label{sec:tv}
The total variation (TV) distance between $\nu, \hat\nu \in {\cal M}(\mathbb{W})$ is given as  $d_{\text{TV}}(\nu,\hat\nu) \defeq \frac{1}{2}\sup_f \vert\int_\mathbb{W} f(z)\nu(dz) - \int_\mathbb{W} f(z) \hat\nu(dz)\vert$, where the supremum is over measurable functions $f:\mathbb{W} \to \mathbb{R}$ with $\norm{f}_\infty \leq 1$ \cite[page 424]{gibbs2002choosing}.  The definition of $d_{\text{TV}}$ implies $\beta(\nu,\hat\nu) \leq 2 d_{\text{TV}}(\nu,\hat\nu)$  for all $\nu, \hat\nu \in {\cal M}(\mathbb{W})$. Hence, \Cref{assum:betametrize} holds with $\psi(t) = 2t$ for $t \in \mathbb{R}_+$. When $\mathbb{W}$ is a finite set, $d_{\text{TV}}$ satisfies \Cref{assum:controlconc} and a closed-form expression for $\epsilon^N_\gamma$ exists \cite[Subsection 4.1]{ramani2022robust}.

\subsection{Hellinger distance}\label{sec:hellinger}
For $\nu,\hat\nu \in {\cal M}(\mathbb{W})$ that are absolutely continuous with respect to a dominating measure $\lambda$, the Hellinger distance between $\nu$ and $\hat\nu$ is defined as $d_\text{Hellinger}(\nu,\hat\nu) \defeq \int_\mathbb{W} \Big(\sqrt{\frac{d\nu}{d\lambda}(z)} - \sqrt{\frac{d\hat\nu}{d\lambda}(z)}\Big)^2 \lambda(dz)$. Here, $\frac{d\nu}{d\lambda}$ and $\frac{d\hat\nu}{d\lambda}$ are the Radon-Nikodyn derivatives of $\nu$ and $\hat\nu$ with respect to $\lambda$, respectively. Inequality (8) in \cite{gibbs2002choosing} establishes that $d_\text{TV}(\nu,\hat\nu) \leq \sqrt{d_\text{Hellinger}(\nu,\hat\nu)}$. Since $\beta(\nu,\hat\nu) \leq 2 d_{\text{TV}}(\nu,\hat\nu)$ from \Cref{sec:tv}, the relationship $\beta(\nu,\hat\nu) \leq 2 \sqrt{d_{\text{Hellinger}}(\nu,\hat\nu)}$ holds for all $\nu,\hat\nu \in {\cal M}(\mathbb{W})$. Therefore, $d_{\text{Hellinger}}$ satisfies \Cref{assum:betametrize} with $\psi(t) = 2 \sqrt{t}$ for $t \in \mathbb{R}_+$. For finite $\mathbb{W}$, $d_{\text{Hellinger}}$  satisfies \Cref{assum:controlconc} and a formula for $\epsilon^N_\gamma$ appears in \cite[Subsection 4.3]{ramani2022robust}.

\subsection{Kullback-Leibler divergence}\label{sec:kl}
The Kullback-Leibler (KL) divergence between $\nu, \hat\nu \in {\cal M}(\mathbb{W})$ is given as $d_{\text{KL}}(\nu,\hat\nu) \defeq \int_\mathbb{W} \log\left(\frac{d\nu}{d\hat\nu}(z)\right)\nu(dz)$ if $\nu$ is absolutely continuous with respect to $\hat\nu$, and $\infty$ otherwise \cite[page 422]{gibbs2002choosing}.  The Pinsker's inequality states that $d_{\text{TV}}(\nu,\hat\nu) \leq \sqrt{d_{\text{KL}}(\nu,\hat\nu)/2}$. Since $\beta(\nu,\hat\nu) \leq 2 d_{\text{TV}}(\nu,\hat\nu)$ from \Cref{sec:tv}, we have $\beta(\nu,\hat\nu) \leq \sqrt{2d_{\text{KL}}(\nu,\hat\nu)}$ for all $\nu, \hat\nu \in {\cal M}(\mathbb{W})$. Therefore, $d_{\text{KL}}$ satisfies \Cref{assum:betametrize} with $\psi(t) = \sqrt{2t}$ for $t \in \mathbb{R}_+$. An approximate version of \Cref{assum:controlconc} for the KL divergence is presented in \cite[Subsection 4.5]{ramani2022robust}.

\subsection{$\chi^2$ distance}\label{sec:chi_squared}
The $\chi^2$ distance between $\nu,\hat\nu \in {\cal M}(\mathbb{W})$ is defined as  $d_{\chi^2}(\nu,\hat\nu) \defeq \int_\mathbb{W} \left(\frac{d\nu}{d\hat\nu}(z) - 1\right)^2 \hat\nu(dz)$ if $\nu$ is absolutely continuous with respect to $\hat\nu$, and $\infty$ otherwise \cite[page 425]{gibbs2002choosing}. From \cite[page 429]{gibbs2002choosing}, we have $d_{\text{TV}}(\nu,\hat\nu)\leq  \sqrt{d_{\chi^2}(\nu,\hat\nu)}/2$. Since $\beta(\nu,\hat\nu) \leq 2 d_{\text{TV}}(\nu,\hat\nu)$ from \Cref{sec:tv}, $\beta(\nu,\hat\nu) \leq \sqrt{d_{\chi^2}(\nu,\hat\nu)}$ for all $\nu, \hat\nu \in {\cal M}(\mathbb{W})$. Therefore, ${d_{\chi^2}}$ satisfies \Cref{assum:betametrize} with $\psi(t) = \sqrt{t}$ for $t \in \mathbb{R}_+$. The $\chi^2$ distance satisfies an approximate version of \Cref{assum:controlconc} \cite[Subsection 4.6]{ramani2022robust}. 

\subsection{Wasserstein distance}\label{sec:wasserstein}
For any arbitrary $z_0 \in \mathbb{W}$, define ${\cal M'}(\mathbb{W}) \defeq  \{\nu \in {\cal M}(\mathbb{W}): \int_\mathbb{W} \rho_\mathbb{W}(z_0,z) \nu(dz) < \infty \}$, where $\rho_\mathbb{W}$ is a metric on $\mathbb{W}$ that metrizes its topology. The Wasserstein distance (or the 1-Wasserstein distance) between $\nu, \hat\nu \in {\cal M'}(\mathbb{W})$ is defined as $d_{\text{Wasserstein}}(\nu,\hat\nu) \defeq \inf_{\theta \in \Theta(\nu,\hat\nu)} \int_{\mathbb{W} \times \mathbb{W}}\rho_\mathbb{W}(z_1,z_2)\theta(dz_1,dz_2)$,  where $\Theta(\nu,\hat\nu)$ is the set of joint distributions with marginals $\nu$ and $\hat\nu$ \cite[Chapter 6]{villani2009optimal}. From the  Kantrovich-Rubinstein duality formula \cite[Remark 6.5]{villani2009optimal}, $d_{\text{Wasserstein}}(\nu,\hat\nu) = \sup_f\big|\int_\mathbb{W} f(z)\nu(dz) - \int_\mathbb{W} f(z)\hat\nu(dz)\big|$, where the supremum is over measurable functions $f:\mathbb{W} \to \mathbb{R}$ with $\norm{f}_L \leq 1$. Therefore, $\beta(\nu,\hat\nu) \leq d_{\text{Wasserstein}}(\nu,\hat\nu)$ for all $\nu, \hat\nu \in {\cal M'}(\mathbb{W})$. Hence, \Cref{assum:betametrize} is satisfied with $\psi(t)=t$ for $t \in \mathbb{R}_+$. The Wasserstein distance satisfies  \Cref{assum:controlconc} under appropriate conditions on $\mathbb{W}$ and the distribution $\nu$. Suppose $\mathbb{W} \subseteq \mathbb{R}^m$ and $\nu$ is such that $\int_{\mathbb{W}} \exp(\lambda \norm{z}^a) \nu(dz) < \infty$ for some $\lambda > 0, a > 1$. Then, from \cite[Theorem 2]{fournier2015rate}
\begingroup
\allowdisplaybreaks
\begin{align}
\label{eqn:wasserstein_concentration}
\nu_\otimes^N\left[d_{\text{Wasserstein}}(\nu,\hat\nu^N) \leq \epsilon\right] & \geq 1 - c_1 b(N,\epsilon)1_{\{\epsilon \leq 1\}} - c_1\exp(-c_2 N\epsilon^a)1_{\{\epsilon > 1\}},\\
\nonumber
\text{where}\\
\nonumber 
b(N,\epsilon) &\defeq
        \begin{cases}
            \exp\left(-c_2N\epsilon^2\right) \text{ if } m = 1\\
            \exp\left(-c_2N\frac{\epsilon^2}{\log^2(2+1/\epsilon)}\right) \text{ if } m =2\\
            \exp\left(-c_2N\epsilon^m\right) \text{ if } m \geq 3.
        \end{cases}
\end{align}
\endgroup
Here, $c_1,c_2$ are positive constants that depend only on $m,a,\int_{\mathbb{W}} \exp(\lambda \norm{z}^a) \nu(dz)$. By setting $c_1 b(N,\epsilon)1_{\{\epsilon \leq 1\}} + c_1\exp(-c_2 N\epsilon^a)1_{\{\epsilon > 1\}}$ to $\gamma \in (0,1)$, the ambiguity set radius for any sample-size $N \geq 1$ and confidence level $1-\gamma \in (0,1)$ is given by
\begingroup
\allowdisplaybreaks
\begin{align}
 \label{eqn:wassestein_radius}
 \epsilon^N_\gamma(\text{Wasserstein}) &\defeq \begin{cases}
   \left(\frac{\log(c_1/\gamma)}{Nc_2}\right)^{1/a} &\text{ if } N < \log(c_1/\gamma)/c_2\\
   \left(\frac{\log(c_1/\gamma)}{Nc_2}\right)^{1/2} &\text{ if } N \geq \log(c_1/\gamma)/c_2 \text{ and } m =1\\
   \quad \tilde\epsilon &\text{ if } N \geq \log(c_1/\gamma)/c_2 \text{ and } m =2\\ 
   \left(\frac{\log(c_1/\gamma)}{Nc_2}\right)^{1/m} &\text{ if } N \geq \log(c_1/\gamma)/c_2 \text{ and } m \geq 3,\\ 
  \end{cases}
\end{align}
\endgroup
where $\tilde\epsilon$ satisfies the equation $\tilde\epsilon/\log(2+1/\tilde\epsilon) = \sqrt{\log(c_1/\gamma)/Nc_2}$. Inequalities similar to \eqref{eqn:wasserstein_concentration} also exist for compact sets in $\mathbb{R}^m$ \cite[Proposition 3.2]{boskos2020data} and compact Polish spaces \cite{weed2019sharp}, among others. By employing a similar approach used to derive \eqref{eqn:wassestein_radius}, we can then obtain closed-form expressions for $\epsilon^N_\gamma(\text{Wasserstein})$ that satisfy \Cref{assum:controlconc} for those specific instances of $\mathbb{W}$ and $\nu$ under consideration. Though Assumptions \ref{assum:betametrize} and \ref{assum:controlconc} were verified for the 1-Wasserstein distance, they  also hold for the $p$-Wasserstein distance with $p \in [1,\infty)$.

Subgaussian distributions on $\mathbb{R}^m$ \cite[Chapter 3]{vershynin2025high} satisfy the light-tailed property   $\int_{\mathbb{R}^m} \exp(\lambda \norm{z}^a) \nu(dz) < \infty$ with $a=2$ and a $\lambda > 0$ that may be distribution-specific. The proof of \cite[Proposition 6.2.1]{vershynin2025high} establishes this light-tailed property for zero mean subgaussian random vectors. Using similar  arguments, it follows that subgaussian random vectors with nonzero mean also satisfy the light-tailed property with $a=2$ and a distribution-specific $\lambda > 0$.

\subsection{ Bounded Lipschitz metric}
\label{sec:beta}
Recalling that this is the $\beta$-metric defined in \eqref{eqn:betametric},   \Cref{assum:betametrize} then trivially holds for this distance with $\psi(t) = t$ for $t \in \mathbb{R}_+$.  \Cref{sec:wasserstein}  established that $\beta(\nu,\hat\nu) \leq d_{\text{Wasserstein}}(\nu,\hat\nu)$ for all $\nu,\hat\nu \in {\cal M}(\mathbb{W})$. Therefore, for all $\epsilon \in [0,\infty)$, we have that $\nu_\otimes^N\left[\beta(\nu,\hat\nu^N) \leq \epsilon \right] \geq \nu_\otimes^N[d_{\text{Wasserstein}}(\nu,\hat \nu^N) \leq \epsilon]$. This implies that \Cref{assum:controlconc} holds for the $\beta$-metric if it holds for the Wasserstein distance. For example, \Cref{sec:wasserstein} established that \Cref{assum:controlconc} holds for $d_\text{Wasserstein}$ if $\mathbb{W} \subseteq \mathbb{R}^m$ and  $\nu$ satisfies $\int_\mathbb{W} \exp(\lambda \norm{z}^a)\nu(dz) < \infty$ for some $a > 1$ and $\lambda > 0$. Hence,  for any sample-size $N \geq 1$ and confidence level $1-\gamma$ with $\gamma \in (0,1)$, the value of $\epsilon^N_\gamma(\text{BL})$ is given as $ \epsilon^N_\gamma(\text{BL}) \defeq \epsilon^N_\gamma(\text{Wasserstein})$, where $\epsilon^N_\gamma(\text{Wasserstein})$ is defined in \eqref{eqn:wassestein_radius}. Since the $\beta$-metric is upper bounded by $2$, we have that $\epsilon^N_\gamma(\text{BL})$ satisfies \Cref{assum:controlconc} for all sufficiently large $N$.  
\subsection{Prokhorov metric}\label{sec:prokhorov}
For any $T \subset \mathbb{W}$ and $\delta > 0$, define $T^\delta \defeq  \big\{z \in \mathbb{W}: \exists y \in T \text{ such that } \rho_\mathbb{W}(z,y) < \delta \big\}$, where $\rho_\mathbb{W}$ is a metric on $\mathbb{W}$ that metrizes its topology. The Prokhorov distance between $\nu, \hat\nu \in {\cal M}(\mathbb{W})$ is given by $d_{\text{Prokhorov}}(\nu,\hat\nu) \defeq \inf \{\delta > 0: \nu(T) \leq  \hat\nu(T^\delta) + \delta \ \forall \text{ Borel sets } T \}$ \cite[Section 11.3]{dudley2018real}. Since $\beta(\nu,\hat\nu) \leq  2 d_{\text{Prokhorov}}(\nu,\hat\nu)$ for all $\nu,\hat\nu \in {\cal M}(\mathbb{W})$ \cite[Corollary 2]{dudley1968distances}, \Cref{assum:betametrize} holds with $\psi(t) = 2t$ for $t \in \mathbb{R}_+$. The $\beta$-metric is also lower bounded by the Prokhorov metric as $d_{\text{Prokhorov}}(\nu,\hat\nu) \leq \sqrt{3\beta(\nu,\hat\nu)/2}$ for all $\nu,\hat\nu \in {\cal M}(\mathbb{W})$ \cite[Corollary 3]{dudley1968distances}.  \Cref{sec:beta} established that $\beta \leq d_{\text{Wasserstein}}$, and hence $d_{\text{Prokhorov}} \leq \sqrt{3 d_{\text{Wasserstein}}/2}$.   Therefore, $\nu_\otimes^N\left[d_{\text{Prokhorov}}(\nu,\hat\nu^N) \leq \epsilon \right] \geq \nu_\otimes^N\left[d_{\text{Wasserstein}}(\nu,\hat \nu^N) \leq 2\epsilon^2/3 \right]$. Hence, \Cref{assum:controlconc} holds for the Prokhorov distance if it holds for the Wasserstein distance. Thus, as done in \Cref{sec:beta}, if $\mathbb{W} \subseteq \mathbb{R}^m$ and $\nu$ satisfies $\int_\mathbb{W} \exp(\lambda \norm{z}^a)\nu(dz) < \infty$ for some $a > 1$ and $\lambda > 0$, then for any sample-size $N \geq 1$ and confidence level $1-\gamma \in (0,1)$, equating $2\epsilon^2/3$ to $\epsilon^N_\gamma(\text{Wasserstein})$ gives us $\epsilon^N_\gamma(\text{Prokhorov}) \defeq \sqrt{\frac{3}{2}\epsilon^N_\gamma(\text{Wasserstein})}$, where $\epsilon^N_\gamma(\text{Wasserstein})$ is defined in \eqref{eqn:wassestein_radius}. Since $d_{\text{Prokhorov}} \leq 1$, we have that $\epsilon^N_\gamma(\text{Prokhorov})$ satisfies \Cref{assum:controlconc} for all sufficiently large $N$.

\section{Conclusions}\label{sec:conclusions}
We studied the data-driven properties of Borel space RMDPs  with data-driven distance-based ambiguity sets. We proposed an axiomatic characterization of the distance function (Assumptions \ref{assum:distancemetrize} and \ref{assum:controlconc}) and established that our RMDPs satisfy the following three data-driven guarantees: (i) convergence of the robust optimal value function and the out-of-sample value function to the true optimal value function; (ii) a guarantee that the robust optimal value function serves as a high probability upper bound on the out-of-sample value function; (iii) a probabilistic convergence rate and sample complexity. We also derived out-of-distribution performance bounds. We identified several well-studied distances that satisfy the conditions imposed on the distance function. We also performed data-driven analyses of empirical MDPs and noted that they do not satisfy certain finite sample performance guarantees.

The axiomatic approach pursued in this paper, despite decoupling the statistical analyses from the computational aspects, implicitly assumes that the decision-maker can solve the data-driven RMDP to obtain the robust optimal value function and the robust optimal policy. In practice, it is computationally challenging to solve RMDPs on arbitrary Borel spaces. One approach to deal with this computational intractability issue is to exploit the properties of the specific distance function along with suitable discretizations of the spaces under consideration, and then analyze the performance loss due to statistical and computational approximations. A rigorous study of this approach is  deferred to future research. Another potential direction for
future research is to extend the analyses in this paper to unbounded cost functions. It would also be interesting to generalize our analyses to investigate the data-driven properties of average cost problems.

\bibliographystyle{abbrv}  
\bibliography{paper}

\pagebreak

\begin{appendix}
\section{List of notations}
\label{appendix:notation}
\begin{table}[!ht]
\begin{tabular}{l|l|l}
Notation                                        & Description                                                                                                     & Reference
\\
\hline
\noalign{\vskip 1mm}    
$\tilde J^{N,\epsilon}$                         & Robust optimal value function                                                                                   &    \eqref{eqn:robustmdp}    \\
$\hat\pi^N$                                     & Robust optimal policy                                                                                           &  p. 3        \\
$\tilde \Phi^{N,\epsilon}$                      & Robust Bellman operator                                                                                         &   \eqref{eqn:robustbellman}        \\
$J^*$                                           & Optimal value function                                                                                          &  \eqref{eqn:mdpoptimality}         \\
$J(\hat\pi^N,\cdot)$                            & Out-of-sample value function                                                                                    &   p. 3        \\
$\Phi, \Phi_\pi$                                & Bellman and evaluation operator                                                                                 & \eqref{eqn:bellman},  \eqref{eqn:eval}          \\
$Q_{t,v}, \tilde Q_{t,v}^{N,\epsilon}$          & $Q$- and robust $Q$-operator                                                                                    &  \eqref{eqn:Q_t}, \eqref{eqn:tilde_Q_t}         \\
${\bf Diff}$                    & \begin{tabular}[l]{@{}l@{}}Quantity to upper bound the difference between\\ robust and true values\end{tabular} &   \eqref{eqn:diff}        \\
$Q_{t,v}^{\text{true}}, \Phi_{t}^{\text{true}}$ & \begin{tabular}[l]{@{}l@{}}$Q$-operator and Bellman operator used in\\ out-of-distribution analyses\end{tabular}           & \eqref{eqn:ood_Q_real}, \eqref{eqn:ood_phi_real}
\end{tabular}
\caption{\small MDP and data-driven RMDP notations.}
\end{table}

\section{Weak continuity of transition kernel}
 \label{appendix:transitioncont}
 Fix a $\hat \nu \in {\cal M}(\mathbb{W}), \epsilon \in [0,\infty)$, and consider the ambiguity set ${\cal P}(\hat\nu, \epsilon) \defeq \{\nu \in {\cal M}(\mathbb{W}):d(\nu,\hat \nu) \leq \epsilon\}$. The transition kernel $Q$ is a stochastic kernel from $\mathbb{K} \times {\cal P}(\hat\nu,\epsilon)$ to $\mathbb{X}$ given by $Q(\cdot|x,a,\nu) \defeq \nu \circ F^{-1}(x,a,\cdot)$ for $(x,a) \in \mathbb{K}$ and $\nu \in {\cal P}(\hat\nu,\epsilon)$. The transition kernel is weakly continuous if $(x,a,\nu) \mapsto \int_\mathbb{X} u(z)Q(dz|x,a,\nu)$ is continuous over $\mathbb{K} \times {\cal P}(\hat\nu,\epsilon)$ for any $u \in C_b(\mathbb{X})$. Here,  ${\cal P}(\hat\nu,\epsilon)$ is endowed with the topology of weak convergence.
 
\begin{lem}[Proposition 6.2 in \cite{gonzalez2003minimax}]\label{lem:transitioncont}
\normalfont
If  \Cref{assum:mdpassum}(b) holds, then the  transition kernel is weakly continuous. 
\end{lem}

\section{Proofs for Section  \ref{sec:data-driven_rmdp}}
\label{appendix:bounded_cost}
\subsection{Proofs for \Cref{sec:valueconv}}
\label{appendix:bounded_cost_asymptotic}
We establish intermediate lemmas that will be used to prove our asymptotic convergence claims.

\begin{lem}\label{lem:bellmanberge}
\normalfont
Suppose \Cref{assum:mdpassum} is satisfied. The following then holds.
\begin{enumerate}[(a)]
 \item Fix a $v \in C_b(\mathbb{X})$. Then $\mathbb{X} \ni x \mapsto (\Phi_t v)(x)$ is continuous and bounded for all $t \in \mathbb{N}$.
 \item $J^*$ is continuous over $\mathbb{X}$ and $\norm{J^*}_\infty \leq \norm{c}_\infty/(1-\alpha)$.
\end{enumerate}
\end{lem}
\begin{proof}[Proof of \Cref{lem:bellmanberge}]
We prove part (a) using induction on $t \in \mathbb{N}$. Consider the base case that corresponds to $t = 1$.  From \Cref{assum:mdpassum}(b),(c), the functions $F$ and $c$ are continuous over $\mathbb{K} \times \mathbb{W}$, and $c$ is bounded over $\mathbb{K} \times \mathbb{W}$. Hence, $(x,a) \mapsto \int_\mathbb{W} \left(c(x,a,z) + \alpha  v(F(x,a,z))\right)\mu(dz)$ is continuous over $\mathbb{K}$ by the dominated convergence theorem. Next, the set-valued mapping $\mathbb{X} \ni x \mapsto A(x)$ is compact-valued and continuous by \Cref{assum:mdpassum}(a). Therefore, the continuity of $\mathbb{X} \ni x \mapsto (\Phi v)(x)$ follows by Berge's maximum theorem \cite[Theorem 17.31]{aliprantis2006infinite}. The boundedness of $\Phi v$ follows by noting that $c$ and $v$ are bounded. This establishes the base case, and the induction step  follows by identical base case arguments. Regarding part (b), from \Cref{thm:bellman}(c), $\{\Phi_t{\bf 0}\}_{t \in \mathbb{N}}$ converges to $J^*$ uniformly over $\mathbb{X}$. From part (a) of this lemma, $\Phi_t{\bf 0}$ is continuous over $\mathbb{X}$ for each $t \in \mathbb{N}$. Hence, $J^*$ is continuous over $\mathbb{X}$. The upper bound on $\norm{J^*}_\infty$ follows from the definition of $J^*$ in \eqref{eqn:mdpoptimality} along with the boundedness of $c$. 
\end{proof}

\begin{lem}
\label{lem:ballconv} 
\normalfont
Suppose the radii of ambiguity sets depend on the sample-size $N$ such that $\lim_{N \to \infty}\epsilon^N = 0$. If \cref{assum:distancemetrize} holds, then $\mu_\otimes^\infty\left[\nu^N \to \mu \text{ weakly } \forall \nu^N\in {\cal P}^N(\epsilon^N)\right] = 1$, where $\mu$ is the true distribution of the disturbance variable.
\end{lem}
\begin{proof}[Proof of \Cref{lem:ballconv}]
We build on the proof arguments in \cite[Lemma 6, Appendix A.1]{ramani2022robust}. That lemma established a similar convergence result for data-driven distance-based ambiguity sets containing distributions defined on finite sample spaces. Since  $\beta$-metric metrizes the topology of weak convergence \cite[Theorem 11.3.3]{dudley2018real}, it suffices to establish $\mu_\otimes^\infty\big[\beta(\nu^N,\mu) \to  0 \ \forall \nu^N\in {\cal P}^N(\epsilon^N)\big] = 1$. By the triangle inequality, $
\beta(\nu^N,\mu) \leq \beta(\nu^N,\hat{\mu}^N) + \beta(\hat{\mu}^N,\mu)$. From \cite[Theorem 11.4.1]{dudley2018real}, 
$\mu_\otimes^\infty\big[\beta(\hat{\mu}^N,\mu) \to 0\big] = 1$. To establish the convergence of $\beta(\nu^N,\hat{\mu}^N)$ to $0$, since $\nu^N \in {\cal P}^N(\epsilon^N)$, we have $d(\nu^N,\hat{\mu}^N) \leq \epsilon^N$. Hence, the nonnegativity of $d$ along with $\lim_{N\rightarrow \infty}\epsilon^N = 0$ ensures that $\lim_{N\rightarrow \infty} d(\nu^N,\hat{\mu}^N) = 0$. \Cref{assum:distancemetrize} then  implies $\beta(\nu^N,\hat{\mu}^N) \to 0$. 
\end{proof}

Our asymptotic convergences proofs rely on the following version of generalized dominated convergence theorem.
\begin{lem}[Theorem 3.5 in 
\cite{langen1981convergence,serfozo1982convergence}]\label{lem:serfozo}
\normalfont
Suppose $S$ is a metric space. Let $\{\mu^k\}, \mu \in {\cal M}(S)$ such that $\mu^k \to \mu$ weakly. Let $\{u^k\}, u$ be real-valued functions over $S$. If $\sup_{k \in \mathbb{N}}\norm{u^k}_\infty < \infty$ and $u^k(x^k) \to u(x) \forall x^k \to x$, then $\int_S u^k(x)\mu^k(dx) \to \int_S u(x)\mu(dx)$.
\end{lem}

\begin{proof}[Proof of \Cref{thm:robustconvergence}]
We use a successive approximation approach to  prove the stronger claim that $\mu_\otimes^\infty \big[\lim_{N \to \infty}\tilde J^{N,\epsilon^N}(x^N) = J^*(x) \ \forall x^N \to x \big] = 1$. Recalling that $\mu_\otimes^\infty$ is defined on $\times_{N=1}^\infty \mathbb{W}$, for any realization $(w^1,w^2,\ldots) \in \times_{N=1}^\infty \mathbb{W}$, we prove that $\lim_{N \to \infty}\tilde J^{N,\epsilon^N}(x^N) = J^*(x) \ \forall x^N \to x$ if $\nu^N \to \mu$ weakly for all sequences $\nu^N \in {\cal P}^N(\epsilon^N)$. Since \Cref{assum:distancemetrize} holds and $\lim_{N \to \infty}\epsilon^N = 0$, \Cref{lem:ballconv} implies $\mu_\otimes^\infty \big[\lim_{N \to \infty}\tilde J^{N,\epsilon^N}(x^N) = J^*(x) \ \forall x^N \to x \big] = 1$.

Let $x^N$ be any sequence in $\mathbb{X}$ such that $x^N \to x \in \mathbb{X}$. For any fixed $t \in \mathbb{N}$, we have for all $N \geq 1$,
\begin{align*}
 &\abs{\tilde J^{N,\epsilon^N}(x^N) - J^*(x)} \leq \abs{\tilde J^{N,\epsilon^N}(x^N) - (\tilde \Phi^{N,\epsilon^N}_t{\bf 0})(x^N)} + \abs{(\tilde \Phi^{N,\epsilon^N}_t{\bf 0})(x^N) - (\Phi_t{\bf 0})(x)} + \\
 &\qquad \qquad \qquad \qquad \qquad \qquad \abs{(\Phi_t{\bf 0})(x) - J^*(x)}\\
 &\qquad \qquad \leq (\norm{c}_\infty \alpha^t)/(1 - \alpha) +  \abs{(\tilde \Phi^{N,\epsilon^N}_t{\bf 0})(x^N) - (\Phi_t{\bf 0})(x)} + (\norm{c}_\infty \alpha^t)/(1 - \alpha),
\end{align*}
where the first and last term in the second inequality follow from \Cref{thm:robustbellman}(c) and \Cref{thm:bellman}(c), respectively.  Hence, $\lim_{N \to \infty}\tilde J^{N,\epsilon^N}(x^N) = J^*(x)$ if $\lim_{N \to \infty}(\tilde \Phi^{N,\epsilon^N}_t{\bf 0})(x^N) \\= (\Phi_t{\bf 0})(x)$ for all $t \in \mathbb{N}$. Using induction, we prove that $\lim_{N \to \infty}(\tilde \Phi^{N,\epsilon^N}_t{\bf 0})(x^N) = (\Phi_t{\bf 0})(x)$ and $\sup_{N \in \mathbb{N}}\norm{\tilde \Phi^{N,\epsilon^N}_t {\bf 0}}_\infty < \infty$ for each $t \in \mathbb{N}$.

For the base case $t=1$, the definition of $\tilde \Phi^{N,\epsilon^N}$ in \eqref{eqn:robustbellman} along with the boundedness of $c$ by \Cref{assum:mdpassum}(c) implies $\sup_{N \in \mathbb{N}}\norm{\tilde \Phi^{N,\epsilon^N} {\bf 0}}_\infty < \infty$. Next, for any $N \geq 1$, we have,
\begingroup
\allowdisplaybreaks
\begin{align*}
&\abs{(\tilde \Phi^{N,\epsilon^N}{\bf 0})(x^N) - (\Phi{\bf 0})(x)} \leq \abs{(\tilde \Phi^{N,\epsilon^N}{\bf 0})(x^N) - (\Phi{\bf 0})(x^N)} + \abs{(\Phi {\bf 0})(x^N) - (\Phi{\bf 0})(x)}\\
&= \Bigg|\min_{a \in A(x^N)}\sup_{\nu^N \in {\cal P}^N(\epsilon^N)}\int_\mathbb{W} c(x^N,a,z) \nu^N(dz) - \min_{a \in A(x^N)}\int_\mathbb{W} c(x^N,a,z) \mu(dz) \Bigg| +\\ 
& \qquad \qquad \qquad \abs{(\Phi {\bf 0})(x^N) - (\Phi{\bf 0})(x)}\\
&\leq \underbrace{\sup_{a \in A(x^N)}\abs{\sup_{\nu^N \in {\cal P}^N(\epsilon^N)}\int_\mathbb{W} c(x^N,a,z) \nu^N(dz) -  \int_\mathbb{W} c(x^N,a,z) \mu(dz)}}_{{\bf (I)}} + \\
&\qquad \qquad \qquad \underbrace{\abs{(\Phi {\bf 0})(x^N) - (\Phi{\bf 0})(x)}}_{{\bf (II)}}.
\end{align*}
\endgroup
We prove that ${\bf (I)}$ and ${\bf (II)}$ converge to $0$ as $N \to \infty$. The convergence of ${\bf (II)}$ to $0$ as $N \to \infty$ follows from \Cref{lem:bellmanberge}(a). Regarding ${\bf (I)}$, suppose for a contradiction, it does not converge to $0$ as $N \to \infty$. Hence, there exists $\delta > 0$,  subsequences $x^{N_k}$ of $x^N$ and $a^{N_k} \in A(x^{N_k})$ such that
\begin{equation}\label{eqn:robustconvtemp1}
\Bigg|\sup_{\nu^{N_k} \in {\cal P}^{N_k}(\epsilon^{N_k})} \int_\mathbb{W} c(x^{N_k},a^{N_k},z) \nu^{N_k}(dz) - \int_\mathbb{W} c(x^{N_k},a^{N_k},z) \mu(dz)\Bigg| > \delta \ \forall k \in \mathbb{N}.
\end{equation}
Since $x^{N_k} \to x$ and the set-valued mapping $x \mapsto A(x)$ is compact-valued and continuous by \Cref{assum:mdpassum}(a), there exists a subsequence $a^{N_{k_i}}$ of $a^{N_k}$ such that $a^{N_{k_i}} \to a \in A(x)$. By the definition of supremum, $\forall i \in \mathbb{N}, \ \exists \ \bar \nu^{N_{k_i}} \in {\cal P}^{N_{k_i}}(\epsilon^{N_{k_i}})$ such that
\begin{align*}
&\Bigg|\sup_{\nu^{N_{k_i}} \in {\cal P}^{N_{k_i}}(\epsilon^{N_{k_i}})}\int_\mathbb{W} c(x^{N_{k_i}},a^{N_{k_i}},z) \nu^{N_{k_i}}(dz) - \int_\mathbb{W} c(x,a,z)\mu(dz)\Bigg|\\
&\qquad < \delta/3 + \Big|\int_\mathbb{W} c(x^{N_{k_i}},a^{N_{k_i}},z) \bar \nu^{N_{k_i}}(dz) - \int_\mathbb{W} c(x,a,z)\mu(dz)\Big| \ \forall i \in \mathbb{N}.
\end{align*}
Since $c$ is continuous over $\mathbb{K} \times \mathbb{W}$ by \Cref{assum:mdpassum}(c), we have $\lim_{i \to \infty}c(x^{N_{k_i}},a^{N_{k_i}},z^i) = c(x,a,z)$ for all $z^i \to z$. Since \Cref{assum:distancemetrize} holds and $\lim_{N \to \infty}\epsilon^N = 0$, by \Cref{lem:ballconv}, $\lim_{i \to \infty}\bar \nu^{N_{k_i}} = \mu$ weakly. Since $c$ is bounded, by \Cref{lem:serfozo} , $\lim_{i \to \infty}\int_\mathbb{W} c(x^{N_{k_i}},a^{N_{k_i}},z) \bar \nu^{N_{k_i}}(dz) = \int_\mathbb{W} c(x,a,z)\mu(dz)$. Hence, for all large enough $i$,
\begin{equation}\label{eqn:robustconvtemp2}
\Bigg|\sup_{\nu^{N_{k_i}} \in {\cal P}^{N_{k_i}}(\epsilon^{N_{k_i}})}\int_\mathbb{W} c(x^{N_{k_i}},a^{N_{k_i}},z) \nu^{N_{k_i}}(dz) - \int_\mathbb{W} c(x,a,z)\mu(dz)\Bigg| < 2\delta/3.
\end{equation}
Further, from the dominated convergence theorem, 
\begin{equation}\label{eqn:robustconvtemp3}
\lim_{i \to \infty}\int_\mathbb{W} c(x^{N_{k_i}},a^{N_{k_i}},z) \mu(dz)  = \int_\mathbb{W} c(x,a,z)\mu(dz).
\end{equation}
From \eqref{eqn:robustconvtemp2} and \eqref{eqn:robustconvtemp3}, we have for all large enough $i$,
\[
\Bigg|\sup_{\nu^{N_{k_i}} \in {\cal P}^{N_{k_i}}(\epsilon^{N_{k_i}})}\int_\mathbb{W} c(x^{N_{k_i}},a^{N_{k_i}},z) \nu^{N_{k_i}}(dz) - \int_\mathbb{W} c(x^{N_{k_i}},a^{N_{k_i}},z) \mu(dz)\Bigg| < \delta,
\]
which contradicts \eqref{eqn:robustconvtemp1}. Therefore, ${\bf (I)}$ converges to $0$, thus establishing the base case. The induction step follows by an argument similar to that of the base case.
\end{proof}

\begin{proof}[Proof of \Cref{cor:robustintegralconv}]
Since Assumptions \ref{assum:mdpassum}, \ref{assum:ambiguitycompact}, and \ref{assum:distancemetrize} hold, from \Cref{thm:robustconvergence}, $\mu_\otimes^\infty\big[\lim_{N \to \infty}\tilde J^{N,\epsilon^N}(x) = J^*(x) \ \forall x \in \mathbb{X}\big] = 1$. Since $c$ is bounded by \Cref{assum:mdpassum}(c), $\sup_{N \in \mathbb{N}}\vert\vert\tilde J^{N,\epsilon^N}\vert\vert_\infty < \infty$. The result then follows by dominated convergence theorem. 
\end{proof}

\begin{proof}[Proof of \Cref{thm:oosconvergence}]
Like in the proof of \Cref{thm:robustconvergence}, we establish the stronger claim that $\mu_\otimes^\infty\big[\lim_{N \to \infty}\\  J(\hat\pi^N,x^N) = J^*(x) \ \forall x^N \to x \big] = 1$. For any realization $(w^1,w^2,\ldots) \in \times_{N=1}^\infty \mathbb{W}$, we prove that $\lim_{N \to \infty}J(\hat\pi^N,x^N) = J^*(x) \ \forall x^N \to x$ if $\nu^N \to \mu$ weakly for all sequences $\nu^N \in {\cal P}^N(\epsilon^N)$. Since \Cref{assum:distancemetrize} holds and $\lim_{N \to \infty}\epsilon^N = 0$, the claim $\mu_\otimes^\infty\big[\lim_{N \to \infty}J(\hat\pi^N,x^N) = J^*(x) \ \forall x^N \to x \big] = 1$ follows by \Cref{lem:ballconv}.

Consider any $x^N \to x$ and $t \in \mathbb{N}$. We have for all $N \geq 1$,
\begin{align*}
J(\hat\pi^N,x^N)-J^*(x) &= J(\hat\pi^N,x^N) - (\Phi_{t,\hat\pi^N}J^*)(x^N) + (\Phi_{t,\hat\pi^N}J^*)(x^N) - J^*(x) \\
&\stackrel{(a)}\leq \alpha^t\left(\norm{c}_\infty/(1-\alpha) + \norm{J^*}_\infty\right) + \abs{(\Phi_{t,\hat\pi^N}J^*)(x^N) - J^*(x)}\\
&\stackrel{(b)}\leq \left(2\norm{c}_\infty\alpha^t\right)/(1-\alpha) + \abs{(\Phi_{t,\hat\pi^N}J^*)(x^N) - J^*(x)}.
\end{align*}
The inequalities in $(a)$ and $(b)$ follow from \Cref{thm:bellman}(c) and \Cref{lem:bellmanberge}(b), respectively. Hence, $\lim_{N \to \infty}J(\hat\pi^N,x^N) = J^*(x)$ if $\lim_{N \to \infty}(\Phi_{t,\hat\pi^N}J^*)(x^N) = (\Phi_t J^*)(x)$ for all $t \in \mathbb{N}$. We first prove that every convergent subsequence of $\hat \pi^N(x^N)$ converges to $a^* \in A(x)$, where $a^*$ is an optimal action in state $x$ for the minimization problem \eqref{eqn:mdpoptimality}. Using this result, the convergence (as $N \to \infty$) of $(\Phi_{t,\hat\pi^N}J^*)(x^N)$ to $(\Phi_t J^*)(x)$ for all $t \in \mathbb{N}$ will be established using induction on $t \in \mathbb{N}$. A similar approach was employed in \cite[Theorem 4.4]{kara2020robustness} to analyze the  convergence of approximate MDP models in the nonrobust setup.

Since $x^N \to x$ and the set-valued mapping $x \mapsto A(x)$ is compact-valued and continuous by \Cref{assum:mdpassum}(a), we have that every convergent subsequence of $\hat \pi^N(x^N)$ converges to an element of $A(x)$.  Let $\hat\pi^{N_k}(x^{N_k})$ be a subsequence of $\hat\pi^N(x^N)$ that converges to an $a^* \in A(x)$. Since $\hat\pi^N$ is a robust optimal policy, by \Cref{thm:robustbellman}(b)
\begingroup
\allowdisplaybreaks
\begin{align}
\label{eqn:oostemp1}
\tilde J^{N_k,\epsilon^{N_k}}(x^{N_k}) &= \sup_{\nu^{N_k} \in {\cal P}^{N_k}(\epsilon^{N_k})}\int_\mathbb{W} \big(c(x^{N_k},\hat\pi^{N_k}(x^{N_k}),z) +\\
\nonumber
&\qquad \qquad \alpha  \tilde J^{N_k,\epsilon^{N_k}}(F(x^{N_k},\hat\pi^{N_k}(x^{N_k}),z))\big)\nu^{N_k}(dz) \quad \forall k \in \mathbb{N}.
\end{align}
\endgroup
Fix a $\delta > 0$. For any $k \in \mathbb{N}$, let $\bar\nu^{N_k} \in {\cal P}^{N_k}(\epsilon^{N_k})$ be a $\delta$-maximizer for the problem in the right-hand side of \eqref{eqn:oostemp1}. Hence,  $\int_\mathbb{W} \big(c(x^{N_k},\hat\pi^{N_k}(x^{N_k}),z) +
\alpha  \tilde J^{N_k,\epsilon^{N_k}}(F(x^{N_k}, \hat\pi^{N_k}(x^{N_k}),z))\big)\bar\nu^{N_k}(dz) \leq \tilde J^{N_k,\epsilon^{N_k}}(x^{N_k}) \leq \int_\mathbb{W} \big(c(x^{N_k},\hat\pi^{N_k}(x^{N_k}),z) + 
\alpha  \tilde J^{N_k,\epsilon^{N_k}}(F(x^{N_k},\hat\pi^{N_k}(x^{N_k}),z))\big)\bar\nu^{N_k}(dz) + \delta \ \forall k \in \mathbb{N}$. From \Cref{thm:robustconvergence} (the  stronger claim established in its proof), $\lim_{k \to \infty}\tilde J^{N_k,\epsilon^{N_k}}(x^{N_k}) = J^*(x)$. Therefore,
\begingroup
\allowdisplaybreaks
\begin{align*}
\liminf_{k \to \infty}\int_\mathbb{W} (c(x^{N_k},\hat\pi^{N_k}(x^{N_k}),z) +
\alpha  \tilde J^{N_k,\epsilon^{N_k}}(F(x^{N_k}, \hat\pi^{N_k}(x^{N_k}),z)))\bar\nu^{N_k}(dz) \leq  J^*(x)\\
\leq \limsup_{k \to \infty}\int_\mathbb{W} (c(x^{N_k},\hat\pi^{N_k}(x^{N_k}),z) + 
\alpha  \tilde J^{N_k,\epsilon^{N_k}}(F(x^{N_k},\hat\pi^{N_k}(x^{N_k}),z)))\bar\nu^{N_k}(dz) + \delta. 
\end{align*}
\endgroup
We first prove the equality of  liminf and limsup in the above expression. Since $c$ is continuous over $\mathbb{K} \times \mathbb{W}$ by \Cref{assum:mdpassum}(c), $c(x^{N_k},\hat\pi^{N_k}(x),z^k) \to c(x,a^*,z)$  for all $z^k \to z$. Since $F$ is continuous over $\mathbb{K} \times \mathbb{W}$ by \Cref{assum:mdpassum}(b), from \Cref{thm:robustconvergence} (the  stronger claim established in its proof), $\lim_{k \to \infty}\tilde J^{N_k,\epsilon^{N_k}}(F(x^{N_k},\hat\pi^{N_k}( x^{N_k}),z^k)) =  J^*(F(x,a^*,z))$ for all $z^k \to z$.
The boundedness of $c$ along with the definition of the robust optimal value function in \eqref{eqn:robustmdp} implies  $\sup_{k \in \mathbb{N}}\big|\big|c + \tilde J^{N_k,\epsilon^{N_k}}\big|\big|_\infty < \infty$. Since $\epsilon^N \to 0$, by \Cref{lem:ballconv}, $\lim_{k \to \infty}\bar\nu^{N_k} = \mu$ weakly. Hence, by  \Cref{lem:serfozo}, $\lim_{k \to \infty}\int_\mathbb{W} (c(x^{N_k},\hat\pi^{N_k}(x^{N_k}),z) + 
\alpha  \tilde J^{N_k,\epsilon^{N_k}}(F(x^{N_k},\hat\pi^{N_k}(x^{N_k}),z)))\bar\nu^{N_k}(dz) = \int_\mathbb{W} (c(x, a^*,z) + \alpha J^*(F(\\x,a^*,z))\mu(dz)$. As $\delta > 0$ is arbitrary, we have $J^*(x) = \int_\mathbb{W}(c(x,a^*,z) + \alpha J^*(F(x,a^*,z)))\mu(dz)$. \Cref{thm:bellman}(b) then implies that $a^*$ is an optimal action in state $x$ for problem \eqref{eqn:mdpoptimality}, i.e., every convergent subsequence of $\hat\pi^N(x^N)$ converges to an element in $A(x)$ that is an optimal action in state $x$ for problem \eqref{eqn:mdpoptimality}.

We now establish the convergence (as $N \to \infty$) of $(\Phi_{t,\hat\pi^N}J^*)(x)$ to $(\Phi_t J^*)(x)$ for all $t \in \mathbb{N}$. Consider the base case $t = 1$. For all $N \geq 1$, we have that $
\big|(\Phi_{\hat\pi^N}J^*)(x^N) - (\Phi J^*)(x)\big| = \big|\int_\mathbb{W}(c(x^N,\hat\pi^N(x^N),z) + \alpha J^*(F(x^N,\hat\pi^N(x^N),z)))\mu(dz) - \min_{a \in A(x)}\int_\mathbb{W}(c(x,a,z) +  \alpha J^*(F(x,a,z)))\mu(dz)\big|.$ Suppose for a contradiction $\abs{(\Phi_{\hat\pi^N}J^*)(x^N) - (\Phi J^*)(x)}$ does not converge to $0$. Hence, there exists $\delta > 0$, subsequences $x^{N_k}$ and $\hat\pi^{N_k}(x^{N_k})$, such that,
\begingroup
\allowdisplaybreaks
\begin{align}
\nonumber
&\Big|\int_\mathbb{W}(c(x^{N_k},\hat\pi^{N_k}(x^{N_k}),z) + \alpha J^*(F(x^{N_k},\hat\pi^{N_k}(x^{N_k}),z)))\mu(dz) -\\
\label{eqn:oostemp3}
&\qquad \qquad \min_{a \in A(x)}\int_\mathbb{W}(c(x,a,z) + \alpha J^*(F(x,a,z)))\mu(dz)\Big| > \delta \ \forall k \in \mathbb{N}.
\end{align}
\endgroup
Since $x^N \to x$ and the set-valued mapping $x \mapsto A(x)$ is compact-valued and continuous,  $\hat\pi^{N_k}(x^{N_k})$ has a convergent subsequence $\hat\pi^{N_{k_i}}(x^{N_{k_i}})$ that converges to an $a^* \in A(x)$. As established earlier in the proof, $a^*$ is an optimal action in state $x$ for  problem \eqref{eqn:mdpoptimality}. Next, by \Cref{lem:bellmanberge}(b), $J^*$ is continuous and bounded over $\mathbb{X}$. Hence, by the dominated convergence theorem,
\begingroup
\allowdisplaybreaks
\begin{align*}
&\lim_{i \to \infty}\int_\mathbb{W}\left(c(x^{N_{k_i}},\hat\pi^{N_{k_i}}(x^{N_{k_i}}),z) + \alpha J^*(F(x^{N_{k_i}},\hat\pi^{N_{k_i}}(x^{N_{k_i}}),z))\right)\mu(dz)\\
&= \int_\mathbb{W}\hspace*{-2mm}\left(c(x,a^*,z) + \alpha J^*(F(x,a^*,z))\right)\mu(dz)\\
&= J^*(x) =  \min_{a \in A(x)}\int_\mathbb{W}\left(c(x,a,z) + \alpha J^*(F(x,a,z))\right)\mu(dz),
\end{align*}
\endgroup
thereby contradicting \eqref{eqn:oostemp3}. Hence, $\abs{(\Phi_{\hat\pi^N}J^*)(x) - (\Phi J^*)(x)} \to 0$ as $N \to \infty$, thereby establishing the base case. The induction step follows by a similar argument as that of the base case.
\end{proof}

\begin{proof}[Proof of \Cref{cor:oosintegralconv}]
Similar to the proof of \Cref{cor:robustintegralconv}, and hence omitted.
\end{proof}
\subsection{Proofs for \Cref{sec:probguarantee}}
\label{appendix:bounded_cost_probguarantee}
\begin{proof}[Proof of \Cref{prop:prob_perf_pointwise}]
 Our approach is similar to the proof in \cite[Theorem 3]{yang2020wasserstein}. We present it here for completeness. From the definition of $\Phi_{\hat\pi^N}$ in \eqref{eqn:eval}, for any $x \in \mathbb{X}$,
 \begingroup
 \allowdisplaybreaks
 \begin{align*}
  &(\Phi_{\hat\pi^N}\tilde J^{N,\epsilon})(x) =  \int_\mathbb{W} \left(c(x,\hat\pi^N(x),z) + \alpha  \tilde J^{N,\epsilon}(F(x,\hat\pi^N(x),z))\right)\mu(dz)\\
  &\quad \leq \sup_{\nu \in {\cal P}^N(\epsilon)}\int_\mathbb{W} \left(c(x,\hat\pi^N(x),z) + \alpha  \tilde J^{N,\epsilon}(F(x,\hat\pi^N(x),z))\right)\nu(dz) = \tilde J^{N,\epsilon}(x),
 \end{align*}
 \endgroup
where the inequality follows since $\mu \in {\cal P}^N(\epsilon)$ and the final equality from  \eqref{eqn:rmdp_policy_bellman}.

By the monotonicity property of the integral and the above  established claim $(\Phi_{\hat\pi^N} \tilde J^{N,\epsilon})(x) \leq  \tilde J^{N,\epsilon}(x)  \ \forall x \in \mathbb{X}$,  we have from the definition of $\Phi_{\hat\pi^N}$ in \eqref{eqn:eval} that $(\Phi_{\hat\pi^N} \Phi_{\hat\pi^N} \tilde J^{N,\epsilon})(x)  \leq (\Phi_{\hat\pi^N} \tilde J^{N,\epsilon})(x) \\ \forall x \in \mathbb{X}$. This implies $(\Phi_{2,\hat\pi^N} \tilde J^{N,\epsilon})(x) \leq (\Phi_{\hat\pi^N} \tilde J^{N,\epsilon})(x) \leq  \tilde J^{N,\epsilon}(x) \  \forall x \in \mathbb{X}$,  where the last inequality follows from our established claim $(\Phi_{\hat\pi^N} \tilde J^{N,\epsilon})  \leq  \tilde J^{N,\epsilon}$. Repeatedly applying $\Phi_{\hat\pi^N}$, for any  $k \in \mathbb{N}$, we get $(\Phi_{k,\hat\pi^N} \tilde J^{N,\epsilon})(x)  \leq  \tilde J^{N,\epsilon}(x)  \ \forall x \in \mathbb{X}$. The claim $J(\hat\pi^N,x) \leq \tilde J^{N,\epsilon}(x) \ \forall x \in \mathbb{X}$ then follows by passing to limit $k \to \infty$ since $\lim_{k \to \infty}(\Phi_{k,\hat\pi^N} \tilde J^{N,\epsilon})(x) = J(\hat\pi^N,x) \ \forall x \in \mathbb{X}$ by  \Cref{thm:bellman}(c).
\end{proof}

\begin{proof}[Proof of \Cref{thm:controlinfiniteprob}]
We have $\mu_\otimes^N[J(\hat\pi^N,x) \leq \tilde J^{N,\epsilon^N_\gamma} \ \forall x \in \mathbb{X}] \geq \mu_\otimes^N[d(\mu,\hat\mu^N) \leq \epsilon^N_\gamma] \geq 1-\gamma$, where the first inequality is due to \Cref{prop:prob_perf_pointwise}, and the second one follows by the choice of $\epsilon^N_\gamma$ and \Cref{assum:controlconc}.\end{proof}

\subsection{Proofs for \Cref{sec:conv_rate}}
\label{appendix:bounded_cost_conv_rate}
Let $\rho_\mathbb{X},\rho_\mathbb{A},\rho_\mathbb{W}$ be the metrics on $\mathbb{X},\mathbb{A}$, and $\mathbb{W}$, respectively, that metrize their corresponding topologies. Our proof of \Cref{thm:conv_rate} requires intermediate lemmas that are presented next.

\begin{lem}\label{lem:phi_t_lipschitz}
\normalfont
Consider any $v: \mathbb{X} \to \mathbb{R}$ which is $L_v$-Lipschitz continuous. Suppose \Cref{assum:conv_rate_assum} holds. For all $t \in \mathbb{N}$, define $\delta_t \defeq L_c \sum_{k=1}^t(\alpha L_F)^{k-1} + L_v \left(\alpha L_F\right)^t$. Then,
\begin{enumerate}[(a)]
\item $\abs{(\Phi_t v)(x_1) - (\Phi_t v)(x_2)} \leq \delta_t\rho_\mathbb{X}(x_1,x_2)$ for all $t \in \mathbb{N}$ and $x_1,x_2 \in \mathbb{X}$.
\item $\abs{(\Phi_t v)(F(x,a,z_1)) - (\Phi_t v)(F(x,a,z_2))} \leq \delta_t L_F \rho_\mathbb{W}(z_1,z_2)$ for all $t \in \mathbb{N}, \ x \in \mathbb{X}, a\in A, z_1,z_2 \in \mathbb{W}$.
\end{enumerate}
\end{lem}
\Cref{lem:phi_t_lipschitz} appears in varying forms in the literature; see  \cite[Theorem 4.37]{saldi2018finite} for a proof of this lemma.

\begin{defn}\label{defn:diff}
\normalfont
 For all $N \in \mathbb{N}, \epsilon \in [0,\infty), t \in \mathbb {N}, v \in \Sigma_b(\mathbb{X})$, define $Q_{t,v}: \mathbb{X} \times \mathbb{A} \to \mathbb{R}$ and $\tilde Q_{t,v}^N: \mathbb{X} \times \mathbb{A} \to \mathbb{R}$ as
 \begingroup
 \allowdisplaybreaks
\begin{align}\label{eqn:Q_t}
Q_{t,v}(x,a) &\defeq \int_\mathbb{W} \left(c(x,a,z) + \alpha (\Phi_{t-1} v)(F(x,a,z))\right)\mu(dz),\\
\label{eqn:tilde_Q_t}
\tilde Q_{t,v}^{N,\epsilon} (x,a) &\defeq \sup_{\nu^N \in {\cal P}^N(\epsilon)}\int_\mathbb{W} \left(c(x,a,z) + \alpha (\tilde \Phi_{t-1}^{N,\epsilon} v)(F(x,a,z))\right)\nu^N(dz).
\end{align}
\endgroup
For $v:\mathbb{X} \to \mathbb{R}$ that is $L_v$-Lipschitz and bounded, define ${\bf Diff}(t,v,N,\epsilon)$ as
\begingroup
\allowdisplaybreaks
\begin{align}
\label{eqn:diff}
&{\bf Diff}(t,v,N,\epsilon) \defeq {\bf Diff}_1(N,\epsilon) {\bf Diff}_2(t,v), \text{ where, }\\
\nonumber
{\bf Diff}_1(N,\epsilon) \defeq \psi(d(\mu,\hat \mu^N)) &+ \psi(\epsilon), \text{ and }\\
\nonumber
{\bf Diff}_2(t,v) \defeq
\Bigg[(\norm{c}_\infty + L_c)&\sum_{k=1}^t \alpha^{k-1} + \alpha^t\left(\norm{v}_\infty + L_v L_F\right)+ \norm{c}_\infty\sum_{j=1}^{t-1}\alpha^j\sum_{k=1}^{t-j}\alpha^{k-1} \\
\nonumber
+ (t-1)\alpha^t\norm{v}_\infty +
 \left(\alpha L_c L_F\right)&\Bigg(\sum_{j=1}^t\alpha^{j-1} \sum_{k=1}^{t-j}(\alpha L_F)^{k-1}\Bigg) + \alpha L_v L_F \sum_{j=1}^{t-1}(\alpha L_F)^j\alpha^{t-j-1}\Bigg].
\end{align}
\endgroup
\end{defn}

\begin{lem}\label{lem:conv_rate_operators}
 \normalfont
Fix a sample-size $N \geq 1, \epsilon \in [0,\infty)$.  Suppose \Cref{assum:conv_rate_assum} holds. Let the distance function $d$ satisfy \Cref{assum:betametrize}, and the function $\psi$ introduced there is nondecreasing. Then, for all $v:\mathbb{X} \to \mathbb{R}$ that is $L_v$-Lipschitz and bounded,
\begingroup
\allowdisplaybreaks
\begin{align}\label{eqn:conv_rate_1}
\big|Q_{t,v}(x,a) - \tilde Q_{t,v}^{N,\epsilon} (x,a)\big| &\leq {\bf Diff}(t,v,N,\epsilon) \quad  \forall t \in \mathbb{N},(x,a) \in \mathbb{X} \times \mathbb{A},\\
\label{eqn:conv_rate_2}
\big|(\Phi_t v)(x) - (\tilde \Phi_t^{N,\epsilon} v)(x)\big| &\leq {\bf Diff}(t,v,N,\epsilon) \quad \forall t \in \mathbb{N}, x \in \mathbb{X}.
\end{align}
\endgroup
\end{lem}
\begin{proof}[Proof of \Cref{lem:conv_rate_operators}]
We first prove that \eqref{eqn:conv_rate_1} implies \eqref{eqn:conv_rate_2}. Towards that end, for any $x \in \mathbb{X}$ and $t \in \mathbb{N}$, observe that $(\Phi_t v)(x) = \min_{a \in \mathbb{A}}Q_{t,v}(x,a)$ and $(\tilde \Phi_t^{N,\epsilon} v)(x) = \min_{a \in \mathbb{A}}\tilde Q_{t,v}^{N,\epsilon}(x,a)$. Hence, $\abs{(\Phi_t v)(x) - (\tilde \Phi_t^{N,\epsilon} v)(x)} \leq  \sup_{a \in \mathbb{A}}\abs{Q_{t,v}(x,a) - \tilde Q_{t,v}^{N,\epsilon} (x,a)}\\ \leq \sup_{a \in \mathbb{A}}{\bf Diff}(t,v,N,\epsilon) = {\bf Diff}(t,v,N,\epsilon)$, thus establishing \eqref{eqn:conv_rate_2}. In the rest of the proof, we prove  inequality \eqref{eqn:conv_rate_1} using induction on $t \in \mathbb{N}$.

Fix $(x,a) \in \mathbb{X} \times \mathbb{A}$ and $\delta > 0$. For any $t \in \mathbb{N}$, let $\nu_t^{N,\epsilon,x,a,\delta} \in {\cal P}^N(\epsilon)$ be a $\delta$-maximizer of $\tilde  Q_{t,v}^{N,\epsilon} (x,a)$. Denote the corresponding value at $\nu_t^{N,\epsilon,x,a,\delta}$ by $\tilde Q_{t,v}^{N,\epsilon,\delta}(x,a)$. Consider the base case $t=1$.  
\begingroup
\allowdisplaybreaks
\begin{align*}
& \abs{Q_{1,v}(x,a) - \tilde Q_{1,v}^{N,\epsilon} (x,a)} \leq \abs{Q_{1,v}(x,a) - \tilde Q_{1,v}^{N,\epsilon,\delta}(x,a)} +  \abs{\tilde Q_{1,v}^{N,\epsilon,\delta}(x,a) - \tilde Q_{1,v}^{N,\epsilon} (x,a)}\\
&\quad \leq \abs{Q_{1,v}(x,a) - \tilde Q_{1,v}^{N,\epsilon,\delta}(x,a)} + \delta\\
&\quad \leq \abs{\int_\mathbb{W} c(x,a,z)\mu(dz) - \int_\mathbb{W} c(x,a,z)\nu_1^{N,\epsilon,x,a,\delta}(dz)} + \\
&\qquad \qquad \alpha \abs{\int_\mathbb{W} v(F(x,a,z))\mu(dz) - \int_\mathbb{W} v(F(x,a,z))\nu_1^{N,\epsilon,x,a,\delta}(dz)} + \delta\\
&\quad \stackrel{(a)} \leq \left(\norm{c}_\infty + L_c\right)\beta(\mu,\nu_1^{N,\epsilon,x,a,\delta}) + \alpha \left(\norm{v}_\infty + L_v L_F\right)\beta(\mu,\nu_1^{N,\epsilon,x,a,\delta}) + \delta\\
&\quad = \Big(\norm{c}_\infty + L_c + \alpha \left(\norm{v}_\infty + L_v L_F\right)\Big)\beta(\mu,\nu_1^{N,\epsilon,x,a,\delta}) + \delta\\
&\quad \stackrel{(b)} \leq \Big(\norm{c}_\infty + L_c + \alpha \left(\norm{v}_\infty + L_v L_F\right)\Big)\left(\beta(\mu,\hat \mu^N) + \beta(\hat\mu^N,\nu_1^{N,\epsilon,x,a,\delta})\right) + \delta\\
&\quad \stackrel{(c)} \leq \Big(\norm{c}_\infty + L_c + \alpha \left(\norm{v}_\infty + L_v L_F\right)\Big)\left(\psi(d(\mu,\hat \mu^N)) + \psi(d(\hat\mu^N,\nu_1^{N,\epsilon,x,a,\delta}))\right) + \delta\\
&\quad \stackrel{(d)} \leq \left(\norm{c}_\infty + L_c\right)\left(\psi(d(\mu,\hat \mu^N)) + \psi(\epsilon)\right) + \alpha \left(\norm{v}_\infty + L_v L_F\right)\left(\psi(d(\mu,\hat \mu^N)) + \psi(\epsilon)\right) + \delta\\
&\quad \stackrel{(e)}= \text{\bf Diff}(1,v,N,\epsilon) + \delta.
\end{align*}
\endgroup
Regarding the first term in the inequality in $(a)$, note that the cost function $c$ is bounded by \Cref{assum:mdpassum}(c) and $c(x,a,\cdot)$ is $L_c$-Lipschitz continuous by \Cref{assum:conv_rate_assum}(b)(i). The first term then follows by the definition of the $\beta$-metric in \eqref{eqn:betametric}. The second term in inequality $(a)$ follows by a similar argument by noting that $F(x,a,\cdot)$ is $L_F$-Lipschitz continuous by \Cref{assum:conv_rate_assum}(b)(ii) and $v$ is $L_v$-Lipschitz and bounded.  Inequality $(b)$ follows by the triangle inequality property of the $\beta$-metric. Inequality $(c)$ holds by \Cref{assum:betametrize}. For the inequality in $(d)$, we have $d(\hat\mu^N,\nu_1^{N,\epsilon,x,a,\delta}) \leq \epsilon$ since $\nu_1^{N,\epsilon,x,a,\delta} \in {\cal P}^N(\epsilon)$. Since $\psi$ is nondecreasing, we have $\psi(d(\hat\mu^N,\nu_1^{N,\epsilon,x,a,\delta})) \leq \psi(\epsilon)$. This, along with a rearrangement of the terms from inequality $(c)$ establishes inequality $(d)$. The equality in $(e)$ then follows from the definition of $\text{\bf Diff}(t,N,v,\epsilon)$ in \eqref{eqn:diff}. The base case follows since $\delta > 0$ is arbitrary.

For the induction step, consider any $t > 1$ and let \eqref{eqn:conv_rate_1} (which in turn implies \eqref{eqn:conv_rate_2}) hold for all $t' \leq t-1$. Then
\begingroup
\allowdisplaybreaks
\begin{align*}
 & \abs{Q_{t,v}(x,a) - \tilde Q_{t,v}^{N,\epsilon} (x,a)} \leq \abs{Q_{t,v}(x,a) - \tilde Q_{t,v}^{N,\epsilon,\delta}(x,a)} + \delta\\
&\quad \leq \abs{\int_\mathbb{W} c(x,a,z)\mu(dz) - \int_\mathbb{W} c(x,a,z)\nu_t^{N,\epsilon,x,a,\delta}(dz)} + \\
&\qquad \quad \alpha \abs{\int_\mathbb{W} (\Phi_{t-1}v)(F(x,a,z))\mu(dz) - \int_\mathbb{W} (\Phi_{t-1}v)(F(x,a,z))\nu_t^{N,\epsilon,x,a,\delta}(dz)} +\\
&\qquad \quad \alpha \abs{\int_\mathbb{W} (\Phi_{t-1}v)(F(x,a,z))\nu_t^{N,\epsilon,x,a,\delta}(dz) - \int_\mathbb{W} (\tilde \Phi_{t-1}^{N,\epsilon} v)(F(x,a,z))\nu_t^{N,\epsilon,x,a,\delta}(dz)} +  \delta\\
&\quad \stackrel{(a)} \leq \left(\norm{c}_\infty + L_c\right)\beta(\mu,\nu_t^{N,\epsilon,x,a,\delta}) +  \alpha \Bigg[\norm{c}_\infty \sum_{k=1}^{t-1}\alpha^{k-1} + \alpha^{t-1}\norm{v}_\infty +\\
&\qquad \quad L_c L_F \sum_{k=1}^{t-1}(\alpha L_F)^{k-1} + L_v L_F \left(\alpha L_F\right)^{t-1} \Bigg]\beta(\mu,\nu_t^{N,\epsilon,x,a,\delta})+ \alpha \text{{\bf Diff}}(t-1,v,N,\epsilon) + \delta\\
&\quad \stackrel{(b)} \leq \Bigg(\norm{c}_\infty + L_c +  \alpha \Bigg[\norm{c}_\infty \sum_{k=1}^{t-1}\alpha^{k-1} + \alpha^{t-1}\norm{v}_\infty + L_c L_F \sum_{k=1}^{t-1}(\alpha L_F)^{k-1} +\\
&\qquad  \qquad L_v L_F \left(\alpha L_F\right)^{t-1} \Bigg]\Bigg)\left(\psi(d(\mu,\hat \mu^N)) + \psi(\epsilon)\right)+ \alpha \text{{\bf Diff}}(t-1,v,N,\epsilon) + \delta\\
&\quad \stackrel{(c)}= \text{\bf Diff}(t,v,N,\epsilon) + \delta.
\end{align*}
\endgroup
The first term in inequality $(a)$ follows by the Lipschitz continuous and boundedness property of $c$. Regarding the second term in inequality $(a)$, first, the  definition of $\Phi_t v$ in \Cref{thm:bellman} implies $\norm{\Phi_t v}_\infty \leq \norm{c}_\infty \sum_{k=1}^t \alpha^{k-1} + \alpha^t \norm{v}_\infty$. This, along with the Lipschitz continuity  of $\Phi_t$  in \Cref{lem:phi_t_lipschitz}(b) results in the second term in $(a)$. The final term in $(a)$ follows by the induction hypothesis. Inequality $(b)$ follows by repeating steps $(b)$-$(d)$ from the proof argument of the base case. Plugging in the definition of ${\bf Diff}(t-1,v,N,\epsilon)$ in inequality $(b)$ along with additional algebra results in equality $(c)$. The induction step follows since  $\delta > 0$ is arbitrary.
\end{proof}

\begin{proof}[Proof of \Cref{thm:conv_rate}]
We first upper bound $\big|J^*(x) - \tilde J^{N,\epsilon}(x)\big|$. For any $x \in \mathbb{X}, t \in \mathbb{N}$,
\begingroup
\allowdisplaybreaks
 \begin{align*}
  &\big|J^*(x) - \tilde J^{N,\epsilon}(x)\big| \\
  &\quad \leq \abs{J^*(x) - (\Phi_t {\bf 0})(x)} + \big|(\Phi_t {\bf 0})(x) - (\tilde \Phi_t^{N,\epsilon} {\bf 0})(x)\big| + \big|(\tilde \Phi_t^N {\bf 0})(x) - \tilde J^{N,\epsilon}(x)\big|\\
  &\quad \leq (\norm{c}_\infty \alpha^t)/(1-\alpha) + {\bf Diff}(t,{\bf 0},N,\epsilon)  + (\norm{c}_\infty \alpha^t)/(1-\alpha),
 \end{align*}
 \endgroup
where the first term in the second inequality follows from \Cref{thm:bellman}(c), the second term from \eqref{eqn:conv_rate_2} in \Cref{lem:conv_rate_operators}, and the final term from \Cref{thm:robustbellman}(c). Hence, $
 \big|J^*(x) - \tilde J^{N,\epsilon}(x)\big| \leq \lim_{t \to \infty}\text{{\bf Diff}}(t,{\bf 0},N,\epsilon) = {\bf \Delta}\big(\psi(d(\mu,\hat \mu^N)) + \psi(\epsilon)\big)$, 
where the last equality follows by the plugging the definition of ${\bf Diff}(t,{\bf 0},N,\epsilon)$ from \eqref{eqn:diff} along with additional algebra.

To prove claim \eqref{eqn:oos_rate}, suppose $d(\mu,\hat\mu^N) \leq \epsilon$.  \Cref{prop:prob_perf_pointwise} then implies $J(\hat\pi^N,x) \leq \tilde J^{N,\epsilon}(x)$ for all $x \in \mathbb{X}$. Hence, $J(\hat\pi^N,x) - J^*(x) \leq \tilde J^{N,\epsilon}(x) - J^*(x) \leq {\bf \Delta}\big(\psi(d(\mu,\hat \mu^N)) + \psi(\epsilon)\big) \leq 2\, \psi(\epsilon){\bf \Delta}$ for all $x \in \mathbb{X}$, where the second inequality follows from the bound established on $\big|J^*(x) - \tilde J^{N,\epsilon}(x)\big|$ and the final inequality by our assumption that $\psi$ is nondecreasing. The claim \eqref{eqn:oos_rate} then follows from  condition \eqref{eqn:concineq_conv}.
\end{proof}

\subsection{Proofs for \Cref{sec:out_of_dist_bounds}}
\label{appendix:bounded_cost_out_of_dist_bounds}
Our proof of \Cref{thm:ood_rate} relies on intermediate bounds that are presented next. Recall from \Cref{sec:out_of_dist_bounds} the quantities associated with the true MDP are superscripted with the label ``true''. For the proxy MDP, we use the same set of notations of the nominal MDP from the previous sections, i.e., $Q_{t,v}, \Phi_t$, etc. The notations for the quantities associated with  data-driven RMDP will be consistent with the previous sections, i.e., $\tilde Q_{t,v}^{N,\epsilon}, \tilde \Phi_t^{N,\epsilon}, \tilde J^{N,\epsilon}, \hat\pi^N$, etc.
\begin{defn}\label{defn:ood_diff}
\normalfont
 Fix $v \in \Sigma_b(\mathbb{X})$. Define $(\Phi^{\text{true}}_0 v)(x) \defeq v(x) \ \forall x \in \mathbb{X}$. For all $t \in \mathbb{N}$, define the functions $Q_{t,v}^{\text{true}}: \mathbb{X} \times \mathbb{A} \to \mathbb{R}$ and $\Phi_t^{\text{true}} v: \mathbb{X} \to \mathbb{R}$ as
 \begingroup
 \allowdisplaybreaks
\begin{align}\label{eqn:ood_Q_real}
Q_{t,v}^{\text{true}}(x,a) &\defeq \int_\mathbb{W} \left(c(x,a,z) + \alpha (\Phi_{t-1}^{\text{true}} v)(F(x,a,z))\right)\mu^{\text{true}}(dz),\\
\label{eqn:ood_phi_real}
(\Phi_{t}^{\text{true}}v)(x) &\defeq \min_{a \in A}Q_{t,v}^{\text{true}}(x,a) \ \forall x \in \mathbb{X}.
\end{align}
\endgroup
\end{defn}

\begin{lem}\label{lem:ood_true_proxy_diff}
 \normalfont
Suppose the components of the true and proxy MDP satisfy  \Cref{assum:conv_rate_assum}. Then, for all $v:\mathbb{X} \to \mathbb{R}$ that is $L_v$-Lipschitz and bounded,
\begingroup
\allowdisplaybreaks
\begin{align}\label{eqn:ood_true_proxy_diff_1}
\abs{Q_{t,v}^{\text{true}}(x,a)-Q_{t,v}(x,a)} &\leq {\bf Diff}_2(t,v) \beta(\mu^{\text{true}},\mu) \quad  \forall t \in \mathbb{N},(x,a) \in \mathbb{X} \times \mathbb{A},\\
\label{eqn:ood_true_proxy_diff_2}
\abs{(\Phi_t^{\text{true}} v)(x)-(\Phi_t v)(x)} &\leq {\bf Diff}_2(t,v) \beta(\mu^{\text{true}},\mu) \quad  \forall t \in \mathbb{N},(x,a) \in \mathbb{X} \times \mathbb{A}.
\end{align}
\endgroup
\end{lem}
\begin{proof}[{\color{black}{Proof of \Cref{lem:ood_true_proxy_diff}}}]
The proof is similar to the proof of  \Cref{lem:conv_rate_operators}. The bound  in \eqref{eqn:ood_true_proxy_diff_2} follows from the bound in \eqref{eqn:ood_true_proxy_diff_1} since for any $x \in \mathbb{X}$ and $t \in \mathbb{N}$, we have $(\Phi_t^{\text{true}} v)(x) = \min_{a \in \mathbb{A}}Q_{t,v}^{\text{true}}(x,a)$ and $(\Phi_t v)(x) = \min_{a \in \mathbb{A}} Q_{t,v}(x,a)$. We now establish  \eqref{eqn:ood_true_proxy_diff_1} using induction on $t \in \mathbb{N}$. For the base case $t=1$, 
  \begin{align*}
   & \abs{Q_{1,v}^{\text{true}}(x,a)-Q_{1,v}(x,a)} \leq \abs{\int_{\mathbb{W}}c(x,a,z)\mu^{\text{true}}(dz) - \int_{\mathbb{W}}c(x,a,z)\mu(dz)} +\\
  &\qquad \alpha \abs{\int_{\mathbb{W}}v(F(x,a,z))\mu^{\text{true}}(dz) - \int_{\mathbb{W}}v(F(x,a,z))\mu(dz)}\\
  &\ \stackrel{(a)} \leq \left(\norm{c}_\infty + L_c\right)\beta(\mu^{\text{true}},\mu) + \alpha \left(\norm{v}_\infty + L_v L_F\right)\beta(\mu^{\text{true}},\mu) \stackrel{(b)}= {\bf Diff}_2(1,v) \beta(\mu^{\text{true}},\mu).
  \end{align*}
Here, $(a)$ follows by the Lipschitz continuity of $c, F, v$ along with the boundedness of $c$ and $v$. The equality in $(b)$ follows from the definition of ${\bf Diff}_2(1,v)$. This proves the base case and the induction step follows by similar base case arguments.
\end{proof}

\begin{lem}\label{lem:ood_conv_rate_operators}
 \normalfont
Fix $N \geq 1, \epsilon \in [0,\infty)$. Let the components of the true and proxy MDP satisfy  \Cref{assum:conv_rate_assum}. Let $d$ satisfy \Cref{assum:betametrize}, and the function $\psi$ stated there is nondecreasing. Let $v:\mathbb{X} \to \mathbb{R}$ be $L_v$-Lipschitz and bounded. For all $t \in \mathbb{N},(x,a) \in \mathbb{X} \times \mathbb{A}$, we have
\begin{align}\label{eqn:ood_conv_rate_1}
\big|Q_{t,v}^{\text{true}}(x,a) - \tilde Q_{t,v}^{N,\epsilon}(x,a)\big| &\leq \big({\bf Diff}_1(N,\epsilon)+ \beta(\mu^{\text{true}},\mu)\big){\bf Diff}_2(t,v), \\
\label{eqn:ood_conv_rate_2}
\big|(\Phi_t^{\text{true}} v)(x) - (\tilde \Phi_t^{N,\epsilon} v)(x)\big| &\leq \big({\bf Diff}_1(N,\epsilon)+ \beta(\mu^{\text{true}},\mu)\big){\bf Diff}_2(t,v).
\end{align}
\end{lem}
\begin{proof}[{\color{black}{Proof of \Cref{lem:ood_conv_rate_operators}}}]
We first prove \eqref{eqn:ood_conv_rate_1}.  
For any $t \in \mathbb{N},(x,a) \in \mathbb{X} \times \mathbb{A}$, we have
$
 \big|Q_{t,v}^{\text{true}}(x,a) - \tilde Q_{t,v}^{N,\epsilon}(x,a)\big| \leq \abs{Q_{t,v}^{\text{true}}(x,a)-Q_{t,v}(x,a)}  +  \big|Q_{t,v}(x,a)-\tilde Q_{t,v}^{N,\epsilon}(x,a)\big|
 \leq {\bf Diff}_2(t,v) \beta(\mu^{\text{true}},\mu) + {\bf Diff}(t,v,N,\epsilon),$ where first term in the last inequality follows from \eqref{eqn:ood_true_proxy_diff_1} and the second term from \eqref{eqn:conv_rate_1}. The claim in \eqref{eqn:ood_conv_rate_1} then follows by plugging the  definition of ${\bf Diff}(t,v,N,\epsilon)$ from \eqref{eqn:diff} in the second term. The claim   \eqref{eqn:ood_conv_rate_2}  follows from \eqref{eqn:ood_conv_rate_1} since for any $x \in \mathbb{X}$ and $t \in \mathbb{N}$, we have $(\Phi_t^{\text{true}} v)(x) = \min_{a \in \mathbb{A}}Q_{t,v}^{\text{true}}(x,a)$ and $(\tilde \Phi_t^{N,\epsilon} v)(x) = \min_{a \in \mathbb{A}}\tilde Q_{t,v}^{N,\epsilon}(x,a)$.
\end{proof}

\begin{proof}[{\color{black}{Proof of \Cref{thm:ood_rate}}}]
We first establish that for any $t \in \mathbb{N}$ and $x \in \mathbb{X}$,
\begin{equation}\label{eqn:ood_rate_temp1}
\textstyle
 (\Phi_t^{\text{true}} \tilde J^{N,\epsilon})(x)  \geq (\Phi^{\text{true}}_{t,\hat \pi^N }\tilde J^{N,\epsilon})(x) - 2\sum_{i=1}^t \alpha^{t-i}
 \big({\bf Diff}_1(N,\epsilon)+ \beta(\mu^{\text{true}},\mu)\big){\bf Diff}_2(i,\tilde J^{N,\epsilon}).
\end{equation}
Before proving \eqref{eqn:ood_rate_temp1}, we justify that using $\tilde J^{N,\epsilon}$ as an argument in ${\bf Diff}_2$ is well-defined. The definition of ${\bf Diff}_2$ in \eqref{eqn:diff} requires that $\tilde J^{N,\epsilon}$ be Lipschitz continuous and bounded. Since the cost function $c$ is bounded, the definition of $\tilde J^{N,\epsilon}$ implies $\big|\big|\tilde J^{N,\epsilon}\big|\big|_\infty \leq \norm{c}_\infty/(1-\alpha)$. Next,  $\tilde J^{N,\epsilon}$ is  $L_c/(1-\alpha L_F)$-Lipschitz continuous. This follows using similar arguments as in the proof of \cite[Lemma 5.1]{minjarez2020zero}, which establishes the Lipschitz continuity of value function for two player stochastic games. We use induction on $t \in \mathbb{N}$ to prove \eqref{eqn:ood_rate_temp1}. Fix $x \in \mathbb{X}$ and consider $t = 1$. Then
\begingroup
\allowdisplaybreaks
\begin{align*}
 &(\Phi^{\text{true}}_1 \tilde J^{N,\epsilon})(x) \stackrel{(a)} \geq (\tilde \Phi_1^{N,\epsilon} \tilde J^{N,\epsilon})(x) - \big({\bf Diff}_1(N,\epsilon)+ \beta(\mu^{\text{true}},\mu)\big){\bf Diff}_2(1,\tilde J^{N,\epsilon}) \\
 &\ = \min_{a \in \mathbb{A}}\tilde Q_{1,\tilde J^{N,\epsilon}}^{N,\epsilon}(x,a) - \big({\bf Diff}_1(N,\epsilon)+ \beta(\mu^{\text{true}},\mu)\big){\bf Diff}_2(1,\tilde J^{N,\epsilon})\\
 &\ \stackrel{(b)} = \tilde Q^{N,\epsilon}_{1,\tilde J^{N,\epsilon}}(x,\hat\pi^N(x)) - \big({\bf Diff}_1(N,\epsilon)+ \beta(\mu^{\text{true}},\mu)\big){\bf Diff}_2(1,\tilde J^{N,\epsilon})\\
 &\ \stackrel{(c)} \geq Q^{\text{true}}_{1,\tilde J^{N,\epsilon}}(x,\hat\pi^N(x)) - 2\big({\bf Diff}_1(N,\epsilon)+ \beta(\mu^{\text{true}},\mu)\big){\bf Diff}_2(1,\tilde J^{N,\epsilon})\\
 &\ = (\Phi^{\text{true}}_{1,\hat\pi^N}\tilde J^{N,\epsilon})(x) - 2\big({\bf Diff}_1(N,\epsilon)+ \beta(\mu^{\text{true}},\mu)\big){\bf Diff}_2(1,\tilde J^{N,\epsilon}).
\end{align*}
\endgroup
Inequality $(a)$ follows from \eqref{eqn:ood_conv_rate_2}. The equality in $(b)$ follows by definition \eqref{eqn:tilde_Q_t} along with the fact that $\hat \pi^N$ is the robust optimal policy. The inequality in $(c)$ follows by \eqref{eqn:ood_conv_rate_1}. This proves the base case, and the induction step follows via similar arguments.

Towards establishing the claim of the theorem, for  any $x \in \mathbb{X}$ and $t \in \mathbb{N}$, we have
\begingroup
\allowdisplaybreaks
\begin{align*}
 \nonumber
 &J^{\text{true}}(\hat\pi^N,x) - J^{*,\text{true}}(x) = 
 [J^{\text{true}}(\hat\pi^N,x)-(\Phi^{\text{true}}_{t,\hat \pi^N }\tilde J^{N,\epsilon})(x)] + \\
 &\qquad \qquad \qquad   [(\Phi^{\text{true}}_{t,\hat \pi^N }\tilde J^{N,\epsilon})(x) - (\Phi^{\text{true}}_t \tilde J^{N,\epsilon})(x)] +[(\Phi^{\text{true}}_t \tilde J^{N,\epsilon})(x) - J^{*,\text{true}}(x)].
\end{align*}
\endgroup
The first and last term in the above expression converge to $0$ uniformly over $\mathbb{X}$ and $\Pi$, as $t \to \infty$. Hence, $J^{\text{true}}(\hat\pi^N,x) - J^{*,\text{true}}(x) = \lim_{t \to \infty}[(\Phi^{\text{true}}_{t,\hat \pi^N }\tilde J^{N,\epsilon})(x) - (\Phi^{\text{true}}_t \tilde J^{N,\epsilon})(x)] \leq \lim_{t \to \infty}2\sum_{i=1}^t \alpha^{t-i}
 \big({\bf Diff}_1(N,\epsilon)+ \beta(\mu^{\text{true}},\mu)\big){\bf Diff}_2(i,\tilde J^{N,\epsilon}) \leq {\bf \Delta}(\psi(d(\mu,\hat\mu^N))+\psi(\epsilon)+\beta(\mu^{\text{true}},\mu))/(1-\alpha)$, where the first inequality follows from \eqref{eqn:ood_rate_temp1} and the final one by using the definition of ${\bf Diff}_1$ and ${\bf Diff}_2$ from \eqref{eqn:diff} along with additional algebra.  The nondecreasing property of $\psi$ along with condition \eqref{eqn:concineq_conv} then establishes the claim of the theorem since $\mu_\otimes^N[\psi(d(\mu,\hat\mu^N)) \leq \psi(\epsilon)] \geq \mu_\otimes^N[d(\mu,\hat\mu^N) \leq \epsilon] \geq 1 -\eta(N,\epsilon)$.
\end{proof}

\section{Proofs for Section  \ref{sec:empirical_mdps}}
\label{appendix:empirical_mdps_counterexample}
\begin{proof}[Proof of \Cref{thm:empirical_mdps_counterexample}]
Consider the MDP with $\mathbb{X} \defeq \{0,1,2,3,4\},\ \mathbb{A} \defeq \{1,3\},  A(x)  \defeq \mathbb{A} \text{ if } x = 0 \text{ and } \{1\} \text{ otherwise}, \ \mathbb{W} \defeq \{0,1\},\ \mu \defeq \text{Bernoulli}(p)$ with $p=0.5$, and $\alpha$ any arbitrary number in $(0,1)$. The function $F:\mathbb{K} \times \mathbb{W} \to \mathbb{X}$ is given as $F(x,a,w) \defeq x+a+w$ if $x = 0, \ a \in A(0), \ w \in \mathbb{W}$; and $x$ for all other $(x,a) \in \mathbb{K} \text{ and } w \in \mathbb{W}$. The single-stage cost function is given by $c(0,1,0) \defeq 12, \ c(0,1,1) \defeq 2, \ c(0,3,0) \defeq 8, \ c(0,3,1) \defeq 4$, and $0$ for all other $(x,a) \in \mathbb{K} \text{ and } w \in \mathbb{W}$. This MDP instance will serve as the true problem.

Let $\pi_1$ and $\pi_2$ be the two deterministic stationary policies of the true MDP, where $\pi_1$ is given as $\pi_1(x) \defeq 1 \text{ if } x = 0 \text{ and } 1 \ \forall x \in \mathbb{X} - \{0\}$, and $\pi_2$ is given as $\pi_2(x) \defeq 3 \text{ if } x = 0 \text{ and } 1 \ \forall x \in \mathbb{X} - \{0\}$. The value of $\pi_1$ is $J(\pi_1,0) = 7$ and $J(\pi_1,x) = 0 \ \text{ for } x \in \mathbb{X} - \{0\}$. The value of $\pi_2$ is $J(\pi_2,0) = 6$ and $J(\pi_2,x) = 0 \ \text{ for } x \in \mathbb{X} - \{0\}$. Note that the optimal policy for the true MDP is $\pi_2$.  Let $\hat p^N \defeq \sum_{i=1}^N w_i/N$ be the empirical estimate of the parameter $p$, where  $w_1,\ldots,w_N$ are iid samples of $w$. The value of $\pi_1$ on the empirical MDP is $\tilde J^N_{\text{Emp}}(\pi_1,0) = 12 - 10 \hat p^N \text{ and } \tilde J^N_{\text{Emp}}(\pi_1,x) = 0 \ \forall x \in \mathbb{X} - \{0\}$. The value of $\pi_2$ on the empirical MDP is $\tilde J^N_{\text{Emp}}(\pi_2,0) = 8 - 4 \hat p^N  \text{ and } \tilde J^N_{\text{Emp}}(\pi_2,x) = 0 \ \forall x \in \mathbb{X} - \{0\}$. 

To prove part (a) of \Cref{thm:empirical_mdps_counterexample}, first consider $N = 1$. If $\hat\pi^N_\text{Emp} = \pi_1$, then $J(\hat\pi^N_\text{Emp},0) - J^*(0) = J(\pi_1,0)-J(\pi_2,0)=1$. Since $\tilde J^N_{\text{Emp}}(\pi_1,0) = 12 - 10 \hat p^N$ and $\tilde J^N_{\text{Emp}}(\pi_2,0) = 8 - 4 \hat p^N$, we will have $\hat\pi^N_\text{Emp} = \pi_1$ if and only if $\hat p^N = 1$, which happens with a probability of $0.5$. Hence, $\hat\pi^N_\text{Emp}$ is not $\delta$-optimal policy for the true MDP with a confidence level of $1-\gamma$ for all $\delta < 1, \ \gamma < 0.5$. A similar analysis for $N = 2$ establishes that $\hat\pi^N_\text{Emp}$ is not $\delta$-optimal policy for the true MDP with a confidence level of $1-\gamma$ for all $\delta < 1, \ \gamma < 0.25$. Combining the conclusions of $N = 1$ and $N = 2$ completes the proof of part (a).

Regarding part (b) of \Cref{thm:empirical_mdps_counterexample}, we first establish that $\{J(\hat\pi^N_\text{Emp},0) \leq \tilde J^{N,*}_{\text{Emp}}(0)\} = \{\hat p^N \leq 1/2\}$. Suppose $\hat p^N \leq 1/2$. Then $\tilde J^{N}_{\text{Emp}}(\pi_2,0) \leq \tilde J^{N}_{\text{Emp}}(\pi_1,0)$, and hence the empirical MDP picks $\pi_2$ as the empirical optimal policy. In that case, $\tilde J^{N,*}_{\text{Emp}}(0) = \tilde J^{N}_{\text{Emp}}(\pi_2,0) = 8 - 4\hat p^N \in [6,8]$ which is greater than $J(\pi_2,0)$. Hence, $\{J(\hat\pi^N_\text{Emp},0) \leq \tilde J^{N,*}_{\text{Emp}}(0)\} \supseteq \{\hat p^N \leq 1/2\}$. For the other direction of inclusion, suppose that $\hat p^N > 1/2$. If $\hat p^N \in (1/2,2/3)$, then $\pi_2$ will be the empirical optimal policy and $\tilde J^{N,*}_{\text{Emp}}(0) = \tilde J^N_{\text{Emp}}(\pi_2,0) = 8 -4\hat p^N \in (5.33,6)$ which is less than $J(\pi_2,0)$. If $\hat p^N \in [2/3,1]$, then $\pi_1$ will be the empirical optimal policy, and we have $\tilde J^{N,*}_{\text{Emp}}(0) = \tilde J^N_{\text{Emp}}(\pi_1,0) = 12 -10\hat p^N \in [2,5.33]$ which is less than $J(\pi_1,0)$. Combining the results of cases $\hat p^N \in (1/2,2/3)$ and $\hat p^N \in [2/3,1]$  proves that $\{J(\hat\pi^N_\text{Emp},0) > \tilde J^{N,*}_{\text{Emp}}(0)\} \supseteq \{\hat p^N > 1/2\}$, establishing the other direction of inclusion. Hence, $\{J(\hat\pi^N_\text{Emp},0) \leq \tilde J^{N,*}_{\text{Emp}}(0)\} = \{\hat p^N \leq 1/2\}$ implies
\[\textstyle
 \mu_\otimes^N[J(\hat\pi^N_\text{Emp},0) \leq \tilde J^{N,*}_{\text{Emp}}(0)] = \mu_\otimes^N[\hat p^N \leq \frac{1}{2}] = \mu_\otimes^N[\sum_{i=1}^N w_i \leq N/2] = 2^{-N}\sum_{i=0}^{\lfloor N/2 \rfloor} \binom N i.
\]
From the binomial theorem, the last expression in the above chain is $0.5$ for positive odd integers. Plugging $N = 2$, the value of the last expression in the chain becomes $0.75$. Using a combinatorial analysis, it follows that $2^{-N}\sum_{i=0}^{\lfloor N/2 \rfloor} \binom N i$ is decreasing over positive even integers and is greater than $0.5$. Hence, part (b) is established.
\end{proof}

\end{appendix}      

\end{document}